\documentclass[10pt]{amsart}

\usepackage{amssymb}
\usepackage{amsthm}
\usepackage{amsmath}

\usepackage[usenames]{color}
\definecolor{ANDREW}{RGB}{255,127,0}

\usepackage{tikz}
\usetikzlibrary{arrows,calc}

\usepackage[foot]{amsaddr}

\usepackage{hyperref}

\theoremstyle{plain}
\newtheorem{proposition}{Proposition}[section]
\newtheorem{theorem}[proposition]{Theorem}
\newtheorem{lemma}[proposition]{Lemma}
\newtheorem{corollary}[proposition]{Corollary}
\theoremstyle{definition}
\newtheorem{example}[proposition]{Example}
\newtheorem{definition}[proposition]{Definition}
\newtheorem{observation}[proposition]{Observation}
\theoremstyle{remark}
\newtheorem{remark}[proposition]{Remark}

\DeclareMathOperator{\Aut}{Aut}

\DeclareMathOperator{\dimension}{dim}

\DeclareMathOperator{\Real}{Re}
\DeclareMathOperator{\Imaginary}{Im}

\DeclareMathOperator{\SL}{SL}
\DeclareMathOperator{\GL}{GL}

\DeclareMathOperator{\SU}{SU}

\DeclareMathOperator{\PGL}{PGL}

\DeclareMathOperator{\End}{End}

\DeclareMathOperator{\Imag}{Im} 
 
\DeclareMathOperator{\Spanset}{Span} 
\DeclareMathOperator{\Isom}{Isom} 
\DeclareMathOperator{\Id}{Id} 
\DeclareMathOperator{\U}{U} 
\DeclareMathOperator{\PU}{PU} 
\DeclareMathOperator{\hol}{hol}

\DeclareMathOperator{\Bc}{\mathcal{B}}

\DeclareMathOperator{\Hc}{\mathcal{H}}

\DeclareMathOperator{\Lc}{\mathcal{L}}

\DeclareMathOperator{\Oc}{\mathcal{O}}

\DeclareMathOperator{\Cb}{\mathbb{C}}

\DeclareMathOperator{\Hb}{\mathbb{H}}

\DeclareMathOperator{\Kb}{\mathbb{K}}
\DeclareMathOperator{\Nb}{\mathbb{N}}
\DeclareMathOperator{\Pb}{\mathbb{P}}
\DeclareMathOperator{\Rb}{\mathbb{R}}

\DeclareMathOperator{\Zb}{\mathbb{Z}}

\newcommand{\abs}[1]{\left|#1\right|}

\newcommand{\norm}[1]{\left\|#1\right\|}

\newcommand{\wh}[1]{\widehat{#1}}
\newcommand{\ip}[1]{\left\langle #1\right\rangle}

\begin{document}

\title[Characterizing the unit ball by its automorphism group]{Characterizing the unit ball by its projective automorphism group}
\author{Andrew M. Zimmer}\address{Department of Mathematics, University of Chicago, Chicago, IL 60637.}
\email{aazimmer@uchicago.edu}
\date{\today}
\subjclass[2010]{}

\begin{abstract} 
In this paper we study the projective automorphism group of domains in real, complex, and quaternionic projective space and present two new characterizations of the unit ball in terms of the size of the automorphism group and the regularity of the boundary. 
\end{abstract}

\maketitle

\section{Introduction}

Suppose $\Kb$ is either the real numbers $\Rb$, the complex numbers $\Cb$, or the quaternions $\Hb$. View $\Kb^{d+1}$ as a right $\Kb$-module and consider the action of $\GL_{d+1}(\Kb)$ on the left. Let $\Pb(\Kb^{d+1})$ be the space of $\Kb$-lines in $\Kb^{d+1}$ (parametrized on the right). Then $\PGL_{d+1}(\Kb)$ acts on $\Pb(\Kb^{d+1})$ by diffeomorphisms. 

Given an open set $\Omega \subset \Pb(\Kb^{d+1})$ the \emph{projective automorphism group} is defined to be
\begin{align*}
\Aut(\Omega) = \left\{ \varphi \in \PGL_{d+1}(\Kb) : \varphi \Omega = \Omega \right\}.
\end{align*}
For instance, consider the set 
\begin{align*}
\Bc = \left\{ [1: z_1 : \dots : z_d] \in \Pb(\Kb^{d+1}) : \sum_{i=1}^d \abs{z_i}^2 < 1\right\}  \subset \Pb(\Kb^{d+1}).
\end{align*}
Then $\Aut(\Bc)$ coincides with the image of $\U_{\Kb}(1,d)$ in $\PGL_{d+1}(\Kb)$ and $\Bc$ is a bounded symmetric domain in the following sense: $\Bc$ is bounded in an affine chart of $\Pb(\Kb^{d+1})$ and $\Aut(\Bc)$ is a simple Lie group which acts transitively on $\Bc$. Moreover, there is a natural $\Aut(\Bc)$-invariant Riemannian metric $g$ which makes $(\Bc, g)$ isometric to $\Kb$-hyperbolic $d$-space (see for instance~\cite[Chapter 19]{M1973}).

The main goal of this paper is to provide new characterizations of this symmetric domain. These characterizations will be in terms of the regularity of the boundary ($\partial \Bc$ is real analytic) and the size of the automorphism group ($\Aut(\Bc)$ acts transitively on $\Bc$). 

We will measure the size of $\Aut(\Omega)$ using the \emph{limit set} $\Lc(\Omega) \subset \partial \Omega$ which is the set of points $x \in \partial \Omega$ so that there exists some $p \in \Omega$ and a sequence $\varphi_n \in \Aut(\Omega)$ with $\varphi_n p \rightarrow x$. Since $\Aut(\Bc)$ acts transitively on $\Bc$ clearly $\Lc(\Bc) = \partial \Bc$.

We will also restrict our attention to a particular class of domains:

\begin{definition}
We call an open set $\Omega \subset  \Pb(\Kb^{d+1})$ a \emph{proper domain} if $\Omega$ is connected and bounded in some affine chart. 
\end{definition}

We first show that $\Bc$ is the only proper domain in complex or quarternionic projective space whose boundary is $C^1$ and whose limit set contains a spanning set.

\begin{theorem}\label{thm:C1}
Suppose $\Kb$ is either $\Cb$ or $\Hb$ and $\Omega \subset \Pb(\Kb^{d+1})$ is a proper domain with $C^1$ boundary. If there exists $x_1, \dots, x_{d+1} \in \Lc(\Omega)$ so that 
\begin{align*}
x_1 + \dots + x_{d+1} = \Kb^{d+1}
\end{align*}
(as $\Kb$-lines) then $\Omega$ is projectively isomorphic to $\Bc$. 
\end{theorem}

\begin{remark}\label{rmk:convex_divisible} Theorem~\ref{thm:C1} fails completely in real projective space. In particular, there are many examples of proper domains $\Omega \subset \Pb(\Rb^{d+1})$ which have $C^1$ boundary, $\Lc(\Omega) = \partial \Omega$, and $\Aut(\Omega)$ is a discrete group which acts properly and co-compactly on $\Omega$.  In some of these examples $\Aut(\Omega)$ is isomorphic to a lattice in $\Isom(\Hb_{\Rb}^d)$ (see~\cite[Section 1.3]{B2000} for $d > 2$ and~\cite{G1990} for $d=2$) while in other examples $\Aut(\Omega)$ is not quasi-isometric to any symmetric space (see~\cite{K2007}). More background on these examples of \emph{divisible sets} in real projective space can be found in the survey papers by Benoist~\cite{B2008}, Goldman~\cite{G2009}, Marquis~\cite{M2013}, and Quint~\cite{Q2010}.
\end{remark}

If $\Aut(\Omega)$ acts co-compactly on $\Omega$ is it straightforward to show that $\Lc(\Omega) = \partial \Omega$ (see Corollary~\ref{cor:cocpct} below). So Theorem~\ref{thm:C1} implies the following:

\begin{corollary}\label{cor:co_cpct}
Suppose $\Kb$ is either $\Cb$ or $\Hb$ and $\Omega \subset \Pb(\Kb^{d+1})$ is a proper domain with $C^1$ boundary. If $\Aut(\Omega)$ acts co-compactly on $\Omega$ then 
$\Omega$ is projectively isomorphic to $\Bc$. 
\end{corollary}

We will show the action of $\Aut(\Omega)$ is proper whenever $\Omega$ is a proper domain (see Proposition~\ref{prop:closed} below). In particular, if $\Aut(\Omega)$ is non-compact then $\Lc(\Omega) \neq \emptyset$. So Theorem~\ref{thm:C1} also implies:

\begin{corollary}
Suppose $\Kb$ is either $\Cb$ or $\Hb$ and $\Omega \subset \Pb(\Kb^{d+1})$ is a proper domain with $C^1$ boundary. If $\Aut(\Omega)$ is non-compact and the group 
\begin{align*}
G=\{ g \in \GL_{d+1}(\Kb) : [g] \in \Aut(\Omega)\}
\end{align*}
acts irreducibly on $\Kb^{d+1}$ then $\Omega$ is projectively isomorphic to $\Bc$. 
\end{corollary}

If $\Omega \subset \Pb(\Kb^{d+1})$ is a proper domain, $\partial \Omega$ is a $C^1$ hypersurface, and $x \in \partial \Omega$ we define $T_x^{\Kb} \partial \Omega \subset \Pb(\Kb^{d+1})$ to be the $\Kb$-hyperplane tangent to $\partial \Omega$ at $x$. It is reasonable to refer to the set $T_x^{\Kb} \partial \Omega \cap \partial \Omega$ as the \emph{closed $\Kb$-face of $x$} in $\partial \Omega$. Our next result shows that $\Bc$ is the only set in projective space with $C^2$ boundary and whose limit set intersects two different closed two faces. 

\begin{theorem}\label{thm:C2}
Suppose $\Kb$ is either $\Rb$, $\Cb$, or $\Hb$ and $\Omega \subset \Pb(\Kb^{d+1})$ is a proper domain with $C^2$ boundary. If there exists $x, y \in \Lc(\Omega)$ with $T_x^{\Kb} \partial \Omega \neq T_y^{\Kb} \partial \Omega$ then $\Omega$ is projectively isomorphic to $\Bc$. 
\end{theorem}

\begin{remark}
Theorem~\ref{thm:C2} fails for domains with $C^{1,1}$ boundary (see Section~\ref{sec:example}) and in the holomorphic setting (see Example~\ref{ex:ellipse}). 
\end{remark}

Using Proposition~\ref{prop:bi_prox_str} below, Theorem~\ref{thm:C2} implies:

\begin{corollary}
Suppose $\Kb$ is either $\Rb$, $\Cb$, or $\Hb$ and $\Omega \subset \Pb(\Kb^{d+1})$ is a proper domain with $C^2$ boundary. If there exists an element $\varphi \in \GL_{d+1}(\Kb)$ which has eigenvalues of different absolute value and $[\varphi] \in \Aut(\Omega)$ then $\Omega$ is projectively isomorphic to $\Bc$. 
\end{corollary}

\subsection*{Acknowledgments} 

This material is based upon work supported by the National Science Foundation under Grant Number NSF 1400919.

\section{Some prior results} 

There is a long history of rigidity results involving the structure of the boundary and the size of $\Aut(\Omega)$. Many previous results make at least one of the following assumptions:
\begin{enumerate}
\item that $\Aut(\Omega)$ (or a discrete subgroup) acts co-compactly on $\Omega$,
\item  $\partial \Omega$ is $C^2$ and satisfies some curvature condition (for instance strong convexity or strong pseudo-convexity), or
\item $\Omega$ is convex.
\end{enumerate}
In this brief section we will survey some of these results in the real projective, the complex projective, and the holomorphic setting. 

\subsection{The real projective setting} As mentioned in Remark~\ref{rmk:convex_divisible} there are many proper domains $\Omega \subset \Pb(\Rb^{d+1})$ with $C^1$ boundary which admit a co-compact action by $\Aut(\Omega)$. However, rigidity appears if one assumes higher regularity. For instance Benoist proved the following characterization of the unit ball in real projective space:

\begin{theorem}\cite[Theorem 1.3]{B2004}
Suppose $\Omega \subset \Pb(\Rb^{d+1})$ is proper convex domain and there exists a discrete group $\Gamma \leq \Aut(\Omega)$ which acts co-compactly on $\Omega$. If $\partial \Omega$ is $C^{1,\alpha}$ for all $\alpha \in [0,1)$ then $\Omega$ is projectively isomorphic to $\Bc$. 
\end{theorem}

Recall that a open bounded set $\Omega \subset \Rb^d$ is a called \emph{strongly convex} if $\Omega = \{ x \in \Rb^d : r(x) < 0\}$ for some $C^2$ function $r: \Rb^d \rightarrow \Rb$ with $\nabla r \neq 0$ near $\partial \Omega$ and 
\begin{align*}
\operatorname{Hess}_x(r)(v,v) > 0
\end{align*}
for all $x \in \partial \Omega$ and $v \in T_x \partial \Omega$. A proper domain $\Omega \subset \Pb(\Rb^{d+1})$ is called \emph{strongly convex} if it is a strongly convex set in some (hence any) affine chart which contains it as a bounded set.  With this terminology Soci{\'e}-M{\'e}thou proved the following rigidity result:

\begin{theorem}\cite{SM2002}
Suppose $\Omega \subset \Pb(\Rb^{d+1})$ is a strongly convex open set. If $\Aut(\Omega)$ is non-compact then 
$\Omega$ is projectively isomorphic to $\Bc$. 
\end{theorem}

\begin{remark} Colbois and Verovic~\cite{CV2002} gave an alternative proof with the additional assumption that $\partial \Omega$ is $C^3$. Later Jo~\cite{J2008} and Yi~\cite{Y2008} proved that it is enough to assume that $\Lc(\Omega)$ contains a point $x$ where $\partial \Omega$ is strongly convex in a neighborhood of $x$. \end{remark}

\subsection{The complex projective setting}

The complex projective setting is more rigid than the real projective setting especially when one assumes that there is a discrete group in $\Gamma \leq \Aut(\Omega)$ which acts co-compactly on $\Omega$.

In $\Pb(\Cb^2)$ there do exist non-symmetric proper domains which admit a co-compact action by a discrete group in $\Aut(\Omega)$. However if $\partial \Omega$ has very weak regularity then a result of Bowen implies that $\Omega$ must be a symmetric domain:

\begin{theorem}\cite{B1979}
Suppose $\Omega \subset \Pb(\Cb^2)$ is a proper domain and $\partial \Omega$ is a Jordan curve with Hausdorff dimension one. If there exists a discrete group $\Gamma \leq \Aut(\Omega)$ which acts co-compactly on $\Omega$ then 
$\Omega$ is projectively isomorphic to $\Bc$. 
\end{theorem}

In $\Pb(\Cb^3)$ the co-compact case is even more rigid and recent work of Cano and Seade implies the following:

\begin{theorem}\cite{CS2014}
Suppose $\Omega \subset \Pb(\Cb^3)$ is a proper domain and $\Gamma \leq \Aut(\Omega)$ is a discrete group which acts co-compactly on $\Omega$. Then
$\Omega$ is projectively isomorphic to $\Bc$. 
\end{theorem}

It is worth noting that Cano and Seade's proof relies on Kobayashi and Ochiai's~\cite{KO1980} classification of compact complex surfaces with a projective structure. 

In higher dimensions we proved the following weaker version of Corollary~\ref{cor:co_cpct}:

\begin{theorem}\cite{Z2013}
Suppose $\Omega \subset \Pb(\Cb^{d+1})$ is a proper $\Cb$-convex domain and there exists a discrete group $\Gamma \leq \Aut(\Omega)$ which acts co-compactly on $\Omega$. If $\partial \Omega$ is $C^{1}$ then $\Omega$ is projectively isomorphic to $\Bc$.
\end{theorem}

\begin{remark}
An open set $\Omega \subset \Pb(\Cb^{d+1})$ is called $\Cb$-convex if its intersection with any complex projective line is simply connected. Surprisingly, this weak form of convexity has strong analytic implications. See~\cite{APS2004, H2007} for more details.
\end{remark}

\subsection{The holomorphic setting} 

There is also a long history of rigidity results involving bounded domains $\Omega \subset \Cb^d$ and their bi-holomorphic automorphism group $\Aut_{\hol}(\Omega)$. We will only mention a few results and refer the reader to the survey articles~\cite{IK1999} and~\cite{K2013} for more details.

The most classical is the well known characterization of the unit ball due to Rosay~\cite{R1979} and Wong~\cite{W1977}. Recall that a bounded domain $\Omega \subset \Cb^d$ is called \emph{strongly psuedo-convex} if $\Omega$ has $C^2$ boundary and the Levi-form at each point in the boundary is positive definite. 

\begin{theorem}[Rosay and Wong Ball Theorem]
Suppose $\Omega \subset \Cb^d$ is a bounded strongly pseudo-convex domain. If $\Aut_{\hol}(\Omega)$ is non-compact then $\Omega$ is bi-holomorphic to the unit ball. 
\end{theorem}

In fact it is enough to assume that the limit set contains a point $x$ where $\partial \Omega$ is  strongly pseudo-convex in a neighborhood of $x$ (see~\cite{R1979}) . Thus one obtains the following characterization of the unit ball:

\begin{corollary}
Suppose $\Omega \subset \Cb^d$ is a bounded domain with $C^2$ boundary. If $\Aut_{\hol}(\Omega)$ acts co-compactly on $\Omega$ then $\Omega$ is bi-holomorphic to the unit ball. 
\end{corollary}

We should also observe the the direct analogue of Theorem~\ref{thm:C2} fails in the holomorphic setting in particular:

\begin{example}\label{ex:ellipse}
Let $\Omega_0 = \{ (z_1, z_2) \in \Cb^2 : \operatorname{Im}(z_1) > \abs{z_2}^4 \}$. Then for $t \in \Rb$, $\Aut(\Omega_0)$ contains the bi-holomorphic map
\begin{align*}
a_t \cdot (z_1, z_2) \rightarrow (e^{4t} z_1, e^t z_2).
\end{align*}
Moreover $\Omega_0$ is bi-holomorphic to $\Omega =  \{ (z_1, z_2) \in \Cb^2 : \abs{z_1}^2+ \abs{z_2}^4 <1 \}$ via the map $F:\Omega_0 \rightarrow \Omega$
\begin{align*}
F(z_1, z_2) = \left( \frac{z_1-i}{z_1+i}, \frac{z_2}{2(z_1+i)^{1/2}} \right).
\end{align*}
Then $b_t = F \circ a_t \circ F^{-1} \in \Aut(\Omega)$ and so $(1,0), (-1,0) \in \Lc(\Omega)$.  
\end{example}

However, we recently proved the following variant of Theorem~\ref{thm:C2} in the complex setting:

\begin{theorem}\cite{Z2015}
Suppose $\Omega \subset \Cb^d$ is a bounded convex open set with $C^\infty$ boundary. If there exists $x, y \in \Lc(\Omega)$ with $T_x^{\Cb} \partial \Omega \neq T_y^{\Cb} \partial \Omega$ then $\Omega$ is bi-holomorphic to a domain of the form 
\begin{align*}
\{ (z_1, \dots,z_d) \in \Cb^d : \abs{z_1}^2 + p(z_2, \dots, z_d) < 1\}
\end{align*}
where $p$ is a polynomial. 
\end{theorem}

\begin{remark} In~\cite{Z2015}, we actually show that $p$ is a \emph{weighted homogeneous polynomial}. \end{remark}

Finally we should mention a remarkable theorem due to Frankel:

\begin{theorem}\cite{F1989}
Suppose $\Omega \subset \Cb^d$ is a bounded convex open set and there exists a discrete group $\Gamma \leq \Aut_{\hol}(\Omega)$ which acts properly discontinuously, freely, and co-compactly on $\Omega$. Then $\Omega$ is a bounded symmetric domain. 
\end{theorem}

\subsection*{Acknowledgments} 

This material is based upon work supported by the National Science Foundation under Grant Number NSF 1400919.

\section{Preliminaries}

\subsection{Notations} Given some object $o$ we will let $[o]$ be the projective equivalence class of $o$, for instance: if $v \in \Kb^{d+1} \setminus \{0\}$ let $[v]$ denote the image of $v$ in $\Pb(\Kb^{d+1})$ and if $\phi \in \GL_{d+1}(\Kb)$ let $[\phi]$ denote the image of $\phi$ in $\PGL_{d+1}(\Kb)$.

For $v,w \in \Kb^{d+1}$ we define the standard inner product 
\begin{align*}
\ip{v,w} = ^t\overline{v}w
\end{align*}
where $ ^t\overline{v}$ is the conjugate transpose of $v$. We let $\norm{v}=\ip{v,v}$ be the norm induced by this inner product and for $T \in \End(\Kb^{d+1})$ let $\norm{T}$ be the associated operator norm. If $\Kb^{(d+1)*}$ is the $\Kb$-module of $\Kb$-linear functions $f: \Kb^{d+1} \rightarrow \Kb$  then we define 
\begin{align*}
\norm{f} = \sup\{ \abs{f(z)} : z \in \Kb^{d+1}, \norm{z}=1\}.
\end{align*}

\subsection{Quarternions} We present a short introduction to the Quaternions in the appendix (Section~\ref{sec:quaternions}).

\section{An intrinsic metric and applications}

Let $\Kb^{(d+1)*}$ denote the $\Kb$-module of $\Kb$-linear functions $f: \Kb^{d+1} \rightarrow \Kb$, that is $f(vz) = f(v)z$ for all $v \in \Kb^{d+1}$ and $z \in \Kb$. Then let $\Pb(\Kb^{(d+1)*})$ be the projective space of lines in $\Kb^{(d+1)*}$ (parametrized on the right). Then the \emph{dual set} of  $\Omega \subset \Pb(\Kb^{d+1})$ is the set
\begin{align*}
\Omega^* = \{ f \in \Pb(\Kb^{(d+1)*}) : \ker f \cap \Omega = \emptyset \}.
\end{align*}
Given $\varphi \in \PGL_{d+1}(\Kb)$ let $^*\varphi \in \PGL(\Kb^{(d+1)*})$ be the transformation $^*\varphi(f) = f \circ \varphi$.  We begin by making some observations:

\begin{observation} \
\begin{enumerate}
\item If $\Omega$ is open then $\Omega^*$ is compact. 
\item If $\Omega$ is bounded in an affine chart then $\Omega^*$ has non-empty interior. 
\item If $\varphi \in \Aut(\Omega)$ then $^*\varphi \in \Aut(\Omega^*)$. 
\end{enumerate}
\end{observation}

Now using the dual we can define a metric which generalizes the classical Hilbert metric in real projective geometry. For an open set $\Omega \subset \Pb(\Kb^{d+1})$ define the function $C_{\Omega}: \Omega \times \Omega \rightarrow \Rb$ by
\begin{align*}
C_{\Omega}(p,q) = \sup_{f,g \in \Omega^*} \log \abs{\frac{f(p)g(q)}{f(q)g(p)}}.
\end{align*}
Since $(\varphi \Omega)^* =  ^*\varphi \Omega^*$, we see that 
\begin{align*}
C_{\Omega}(p,q) = C_{\varphi\Omega}(\varphi p, \varphi q)
\end{align*}
for all $\varphi \in \PGL_{d+1}(\Kb)$ and $p,q \in \Omega$. Thus $C_{\Omega}$ will be $\Aut(\Omega)$-invariant. 

When $\Kb = \Rb$ and $\Omega \subset \Pb(\Rb^{d+1})$ is a convex subset, this function $C_{\Omega}$ coincides with the classical Hilbert metric (see for instance~\cite{K1977}). For $\Kb = \Cb$ this function was introduced by Dubois~\cite{D2009} for linearly convex sets in $\Pb(\Cb^{d+1})$. For such domains, Dubois proved that $C_{\Omega}$ is a complete metric. Additional properties of the metric $C_{\Omega}$ for linearly convex sets were established in~\cite{Z2013}. Finally we recently constructed an analogue of the metric $C_\Omega$ for certain domains in real flag manifolds~\cite{Z2015b}. 

Next we will show that $C_{\Omega}$ is a metric generating the standard topology whenever the domain is proper. However,  without convexity assumptions $C_\Omega$ may not be a complete metric.

\begin{proposition}\label{prop:metric}
Suppose $\Omega \subset \Pb(\Kb^{d+1})$ is a proper domain. Then $C_{\Omega}$ is an $\Aut(\Omega)$-invariant metric on $\Omega$ which generates the standard topology. 
\end{proposition}

It will be helpful to observe that $C_\Omega$ on the unit ball is actually the symmetric metric:

\begin{lemma}\label{lem:symmetric}
If 
\begin{align*}
\Bc = \left\{ [1:z_1:\dots :z_d] \in \Pb(\Kb^{d+1}) : \sum_{i=1}^d \abs{z_i}^2 < 1\right\}
\end{align*}
then $(\Bc, C_{\Bc})$ coincides with the model of  $\Hb_{\Kb}^d$ described in Chapter 19 of~\cite{M1973}. In particular, $C_{\Bc}$ is a complete metric on $\Bc$ which generates the standard topology.
\end{lemma}

\begin{proof}
Let $d_{\Kb}$ be the distance on $\Bc$ described in Chapter 19 of~\cite{M1973} and for $t \in (-1,1)$ let 
\begin{align*}
x_t = [1: t:0: \dots :0].
\end{align*}
Now for any $p, q\in \Omega$ there exists $\varphi\in \SU_{\Kb}(1,d)$ so that $\varphi p = x_0$ and $\varphi_q = x_t$ for some $t \in (0,1)$. Then since $\SU_{\Kb}(1,d)$ acts by isometries on $(\Bc, C_{\Bc})$ and $(\Bc, d_{\Kb})$ it is enough to show 
\begin{align*}
C_{\Bc}(x_0, x_t) = d_{\Kb}(x_0, x_t).
\end{align*}
Moreover when $t \in (0,1)$
\begin{align*}
d_{\Kb}(x_0, x_t) = \log \frac{1+t}{1-t}.
\end{align*}
Using the standard inner product we can identify $\Kb^{(d+1)*}$ with $\Kb^{d+1}$ and then view $\Bc^*$ as a subset of $\Pb(\Kb^{d+1})$. Then 
\begin{align*}
\Bc^* = \left\{ [1:f_1: \dots : f_d] : \sum_{i=1}^d \abs{f_i}^2 < 1\right\}.
\end{align*}
Then if $t \in (0,1)$
\begin{align*}
C_{\Omega}(x_0, x_t) = \sup_{f,g \in \Bc^*} \log \frac{1-tf_1}{1-tg_1}
\end{align*}
and this is clearly maximized when $f=[1:-1:0:\dots:0]$ and $g=[1:1:0:\dots:0]$. So 
\begin{align*}
C_{\Omega}(x_0, x_t) =  \log \frac{1+t}{1-t}
\end{align*}
and thus $C_\Omega = d_{\Kb}$. 
\end{proof}

\begin{proof}[Proof of Proposition~\ref{prop:metric}]
Since $^*\varphi \Omega^*=\Omega^*$ for all $\varphi \in \Aut(\Omega)$, it is clear that $C_{\Omega}$ is $\Aut(\Omega)$-invariant.

Suppose that $p,q,r \in \Omega$. Since $\Omega^*$ is compact there exists $f,g \in \Omega^*$ 
\begin{align*}
C_{\Omega}(p,q) =  \log \abs{\frac{f(p)g(q)}{f(q)g(p)}}.
\end{align*}
Then for $r \in \Omega$ 
\begin{align*}
C_{\Omega}(p,q) =  \log \abs{\frac{f(p)g(q)}{f(q)g(p)}} \leq \log \abs{\frac{f(p)g(r)}{f(r)g(p)}}+\log \abs{\frac{f(r)g(q)}{f(q)g(r)}} \leq C_{\Omega}(p,r)+C_{\Omega}(r,q).
\end{align*}
So $C_{\Omega}$ satisfies the triangle inequality.

Now fix an affine chart $\Kb^d$ which contains $\Omega$ as a bounded set. Then after rescaling we may assume that 
\begin{align*}
\Omega \subset \Bc := \{ z \in \Kb^d : \norm{z} < 1\}.
\end{align*}
By the above Lemma $C_{\Bc}$ is a complete metric which generated the standard topology on $\Bc$. Moreover $\Bc^* \subset \Omega^*$ and so $C_{\Bc} \leq C_{\Omega}$ on $\Omega$. Then for $p,q \in \Omega$ distinct we have
\begin{align*}
0 < C_{\Bc}(p,q) \leq C_{\Omega}(p,q).
\end{align*}
Thus $C_{\Omega}$ is a metric. 

Since $\Omega^*$ is compact the function $C_{\Omega} : \Omega \times \Omega \rightarrow \Rb_{\geq 0}$ is continuous. Thus to show that $C_{\Omega}$ generates the standard topology it is enough to show: for any $p \in \Omega$ and $U \subset \Omega$ an open neighborhood of $p$ there exists $\epsilon > 0$ so that 
\begin{align*}
\{ q \in \Omega : C_{\Omega}(p,q) < \epsilon \} \subset U.
\end{align*}
But since $C_{\Bc}$ generates the standard topology on $\Bc$, there exists $\epsilon > 0$ so that 
\begin{align*}
\{ q \in \Bc : C_{\Bc}(p,q) < \epsilon \} \subset U.
\end{align*}
But then 
\begin{align*}
\{ q \in \Omega : C_{\Omega}(p,q) < \epsilon \} \subset \{ q \in B_R : C_{\Bc}(p,q) < \epsilon \} \subset U
\end{align*} 
since $C_{\Bc} \leq C_{\Omega}$ on $\Omega$. So $C_{\Omega}$ generates the standard topology. 
\end{proof}

\subsection{The automorphism group}

\begin{proposition}\label{prop:closed}
Suppose $\Omega \subset \Pb(\Kb^{d+1})$ is a proper domain. Then $\Aut(\Omega) \leq \PGL_{d+1}(\Kb)$ is a closed subgroup and acts properly on $\Omega$.
\end{proposition}

\begin{proof}
We first show that $\Aut(\Omega)$ is a closed subgroup of $\PGL_{d+1}(\Kb)$. Suppose that $\varphi_n \in \Aut(\Omega)$ and $\varphi_n \rightarrow \varphi$ in $\PGL_{d+1}(\Kb)$. Let $d_{\Pb}$ be a distance on $\Pb(\Kb^{d+1})$ induced by a Riemannian metric. Then since $\varphi_n \rightarrow \varphi$ there exists some $M \geq 1$ so that 
\begin{align*}
\frac{1}{M} d_{\Pb}(p,q) \leq d_{\Pb}(\varphi_n p, \varphi_n q) \leq M d_{\Pb}(p,q)
\end{align*}
for all $p,q \in \Pb(\Kb^{d+1})$ and $n \in \Nb$. Next define $\delta_{\Omega} : \Omega \rightarrow \Rb_{>0}$ to be 
\begin{align*}
\delta_{\Omega}(p) = \inf \{ d_{\Pb}(p,x) : x \in \Pb(\Kb^{d+1}) \setminus \Omega \}.
\end{align*}
Then 
\begin{align*}
\frac{1}{M} \delta_{\Omega}(p) \leq \delta_{\Omega}(\varphi p) \leq M \delta_{\Omega}(p).
\end{align*}
for $p \in \Omega$. So $\varphi(\Omega) \subset \Omega$. Since $\varphi_n^{-1} \rightarrow \varphi^{-1}$ the same argument shows that $\varphi^{-1}(\Omega) \subset \Omega$. Thus $\varphi(\Omega) = \Omega$ and $\varphi \in \Aut(\Omega)$.

We now show that $\Aut(\Omega)$ acts properly. This argument requires some work because $C_{\Omega}$ may not be complete. So suppose $K \subset \Omega$ is a compact subset and $\varphi_n k_n \in K$ for some $\varphi_n \in \Aut(\Omega)$ and $k_n \in K$. We claim that a subsequence of $\varphi_n$ converges in $\Aut(\Omega)$. By passing to a subsequence we can suppose that $k_n \rightarrow k \in K$. Now since $C_{\Omega}$ is a locally compact metric (it generates the standard topology) and $K \subset \Omega$ is compact there exists some $\delta > 0$ so that the set 
\begin{align*}
K_1 = \{ q \in \Omega: C_{\Omega}(K,q) \leq 2\delta\}
\end{align*}
is compact. Next let 
\begin{align*}
K_2 = \{ q \in \Omega : C_{\Omega}(k,q) \leq \delta\}.
\end{align*}
Then for large $n$ we have $\varphi_n(K_2) \subset K_1$. Since $\varphi_n$ preserves the metric $C_{\Omega}$ we can pass to a subsequence and assume that $\varphi_n|_{K_2}$ converges uniformly to a function $f:K_2 \rightarrow K_1$. Moreover
\begin{align*}
C_{\Omega}(f(p_1), f(p_2)) = \lim_{n \rightarrow \infty} C_\Omega(\varphi_n p_1, \varphi_n p_2) = C_{\Omega}(p_1,p_2)
\end{align*}
for all $p_1, p_2 \in \Omega$. Since $C_\Omega$ is a metric, $f$ is injective. Next pick $\wh{\varphi}_n \in \GL_{d+1}(\Kb)$ so that $\norm{\wh{\varphi}_n}=1$. By passing to a subsequence we may assume that $\wh{\varphi}_n \rightarrow \Phi$ in $\End(\Kb^{d+1})$. Moreover, if $p \in K_2 \setminus \ker \Phi$ then 
\begin{align*}
f(p) = \lim_{n \rightarrow \infty} \varphi_n p = \Phi(p).
\end{align*}
Since $K_2$ has non-empty interior and $f$ is injective this implies that $\Phi$ induces an injective map $\Pb(\Kb^{d+1}) \rightarrow \Pb(\Kb^{d+1})$. Hence $\Phi \in \GL_{d+1}(\Kb)$. Thus $\varphi_n \rightarrow [\Phi]$ in $\PGL_{d+1}(\Kb)$ and since $\Aut(\Omega)$ is closed we see that $[\Phi] \in \Aut(\Omega)$. 
\end{proof}

\subsection{The asymptotic geometry of the intrinsic metric}

\begin{proposition}
\label{prop:bd_behav}
Suppose $\Omega \subset \Pb(\Kb^{d+1})$ is a proper domain, $p_n, q_n \subset \Omega$ are sequences  such that $p_n \rightarrow x \in \partial \Omega$, $q_n \rightarrow y \in \partial \Omega$, and 
\begin{align*}
\lim_{n \rightarrow \infty} C_{\Omega}(p_n,q_n) < \infty.
\end{align*}
Then 
\begin{align*}
y \in \cap\{ \ker f : f \in \Omega^*, f(x)=0\}.
\end{align*}
\end{proposition}

\begin{proof}
Suppose $f \in \Omega^*$ and $f(x)=0$. Since $\Omega^*$ has non-empty interior there exists $g \in \Omega^*$ so that $g(x) \neq 0$ and $g(y) \neq 0$. Then
\begin{align*}
C_{\Omega}( p_n, q_n ) \geq \log \abs{ \frac{f(q_n)g(p_n)}{f(p_n)g(q_n)}}.
\end{align*}
Let $\hat{p}_n,\hat{q}_n,\hat{x},\hat{y} \in \Kb^{d+1}$ and $\hat{f},\hat{g} \in \Kb^{(d+1)*}$ be representatives of $p_n,q_n,x,y \in \Pb(\Kb^{d+1})$ and $f,g \in \Pb(\Kb^{(d+1)*})$ normalized such that 
\begin{align*}
\norm{\hat{f}}=\norm{\hat{g}}=\norm{\hat{p}_n}=\norm{\hat{q}_n}=\norm{\hat{x}}=\norm{\hat{y}}=1.
\end{align*}
Then 
\begin{align*}
C_{\Omega}( p_n, q_n ) \geq \log \abs{\frac{\hat{f}(\hat{q}_n)}{\hat{f}(\hat{p}_n)}}+\log \abs{ \frac{\hat{g}(\hat{p}_n)}{\hat{g}(\hat{q}_n)}}.
\end{align*}
Since $f(x)=0$, we see that $\hat{f}(\hat{p}_n) \rightarrow 0$. Since $g(x) \neq 0$ and $g(y) \neq 0$, we see that 
\begin{align*}
\log  \abs{\frac{\hat{g}(\hat{p}_n)}{\hat{g}(\hat{q}_n)}}
\end{align*}
 is bounded from above and below. Thus we must have that $\hat{f}(\hat{q}_n) \rightarrow 0$ and so $y \in \ker f$.
\end{proof}

\begin{proposition}
Suppose $\Omega \subset \Pb(\Kb^{d+1})$ is a proper domain and $p_n, q_n \subset \Omega$ are sequences  such that $p_n \rightarrow x \in \overline{\Omega}$. If 
\begin{align*}
 \lim_{n \rightarrow \infty} C_{\Omega}(p_n, q_n) = 0
 \end{align*}
 then $q_n \rightarrow x$.
 \end{proposition}
 
 \begin{proof}
 Fix an affine chart $\Kb^d$ which contains $\Omega$ as a bounded set. Then after scaling we may assume that
\begin{align*}
\overline{\Omega} \subset \Bc = \{z \in \Cb^d : \norm{z} < 1\}.
\end{align*}
By Lemma~\ref{lem:symmetric},  $C_{\Bc}$ is a complete metric which generated the standard topology on $\Bc$. Moreover $\Bc^* \subset \Omega^*$ and so $C_{\Bc} \leq C_{\Omega}$ on $\Omega$. Then 
\begin{align*}
 \lim_{n \rightarrow \infty} C_{\Bc}(p_n, q_n) = 0
 \end{align*}
and so $q_n \rightarrow x$.
\end{proof}

\begin{corollary}\label{cor:cocpct}
Suppose $\Omega \subset \Pb(\Kb^{d+1})$ is a proper domain and $\Aut(\Omega)$ acts co-compactly on $\Omega$. Then $\Lc(\Omega) = \partial \Omega$. 
\end{corollary}

\begin{proof}
Fix $x \in \partial \Omega$ and a sequence $p_n \in \Omega$ so that $p_n \rightarrow x$. Now there exists a compact set $K \subset \Omega$ and $\varphi_n \in \Aut(\Omega)$ so that $\varphi_n p_n \in K$. We can pass to a subsequence so that $\varphi_n p_n \rightarrow k \in K$. Then 
\begin{align*}
\lim_{n \rightarrow \infty} C_{\Omega}(p_n, \varphi_n^{-1} k) = \lim_{n \rightarrow \infty} C_{\Omega}(\varphi_np_n, k) =0
\end{align*}
and so $\varphi_n^{-1} k \rightarrow x$ by the previous Proposition. 
\end{proof}

\section{Limits of Automorphisms}

\begin{proposition}\label{prop:limits}
Suppose $\Omega \subset\Pb(\Kb^{d+1})$ is proper domain with $C^1$ boundary, $\varphi_n \in \Aut(\Omega)$, and 
\begin{align*}
\varphi_n p \rightarrow x^+ \text{ and } \varphi_n^{-1} p \rightarrow x^-
\end{align*}
where $p \in \Omega$ and $x^+, x^- \in \partial \Omega$. Then 
\begin{enumerate}
\item $\varphi_n q \rightarrow x^+$ and $\varphi_n^{-1} q \rightarrow x^-$ for all $q \in \Omega$, 
\item there exists $f^{\pm} \in \Omega^*$ so that $\ker f^{\pm} = T_{x^{\pm}}^{\Kb} \partial \Omega$, 
\item if $\Phi \in \Pb(\End(\Kb^{d+1}))$ is the element with $\operatorname{Im}(\Phi) = x^+$ and $\ker \Phi = T_{x^-}^{\Kb} \partial \Omega$ then $\varphi_n \rightarrow \Phi$ as elements of $\Pb(\End(\Kb^{d+1}))$, 
\item if $U$ is a neighborhood of $\overline{\Omega} \cap T_{x^-}^{\Kb} \partial \Omega$ and $V$ is a neighborhood of $x^+$ then there exists $N \geq 0$ so that 
\begin{align*}
\varphi_n (\overline{\Omega} \setminus U) \subset V
\end{align*}
for all $n \geq N$. 
\end{enumerate}
\end{proposition}

\begin{proof}
Notice that $\varphi_n \rightarrow \infty$ in $\PGL_{d+1}(\Kb)$ since $x^+,x^- \in \partial \Omega$ and $\Aut(\Omega) \leq \PGL_{d+1}(\Kb)$ is closed. 

We begin by proving part (3). Since $\Pb(\Kb^{d+1})$ is compact it is enough to show that any convergent subsequence of $(\varphi_n)_{n \in \Nb}$ converges to $\Phi$. So, by passing to a subsequence,  we can assume that $\varphi_n$ converges. Then let $\wh{\varphi}_n \in \GL_{d+1}(\Kb)$ be a representative of $\varphi_n$ so that $\norm{\varphi_n} =1$ and $\wh{\varphi}_n \rightarrow \Phi_+$ in $\End(\Kb^{d+1})$.  We can write 
\begin{align*}
\wh{\varphi}_n = k_{n,1} \begin{pmatrix} a_{n,1} & & \\ & \ddots & \\ & & a_{n,d+1} \end{pmatrix} k_{n,2} 
\end{align*}
for some $k_{n,1}, k_{n,2} \in \U_{\Kb}(d+1)$ and $1=a_{n,1} \geq \dots \geq a_{n,d+1}$. By passing to a subsequence we can suppose that $k_{n,1} \rightarrow k_1$, $k_{n,2} \rightarrow k_2$ in $\U_{\Kb}(d+1)$ and the limits
\begin{align*}
\lambda_i^+ := \lim_{n \rightarrow \infty} a_{n,i}, \text{ and } \lambda_i^- := \lim_{n \rightarrow \infty} \frac{a_{n,d+1}}{a_{n,i}}
\end{align*}
exist for $1 \leq i \leq d+1$. Then
\begin{align*}
\Phi_+ = \lim_{n \rightarrow \infty} \wh{\varphi}_{n} = k_1  \begin{pmatrix} \lambda_1^+ & & \\ & \ddots & \\ & & \lambda_{d+1}^+ \end{pmatrix} k_2.
\end{align*}
Now $\wh{\varphi}_{n,-}:= a_{n,d+1} \wh{\varphi}_{n}^{-1}$ is a representative of $\varphi_{n}^{-1}$ which converges in $\End(\Kb^{d+1})$ to 
\begin{align*}
\Phi_- := k_2^{-1}  \begin{pmatrix} \lambda_1^- & & \\ & \ddots & \\ & & \lambda_{d+1}^- \end{pmatrix} k_1^{-1}.
\end{align*}

Next identify $\Kb^{(d+1)*}$ with $\Kb^{d+1}$ using the standard inner product and using this identification view $\Omega^*$ as a subset of $\Pb(\Kb^{d+1})$. Then with this identification
\begin{align*}
\Aut(\Omega^*) = \{ ^t\overline{\varphi} : \varphi \in \Aut(\Omega) \}
\end{align*}
where $ ^t\overline{\varphi} \in \PGL_{d+1}(\Kb)$ is the standard conjugate transpose of $\varphi \in \PGL_{d+1}(\Kb)$. 

Now $\wh{\psi}_{n,+}:= a_{n,d+1} (^t\overline{\wh{\varphi}}_{n}^{-1})$ is a representative of $^t\overline{\varphi}_{n}^{-1}$ which converges in $\End(\Kb^{d+1})$ to 
\begin{align*}
\Psi_+ := k_1  \begin{pmatrix} \lambda_1^- & & \\ & \ddots & \\ & & \lambda_{d+1}^- \end{pmatrix} k_2
\end{align*}
and $\wh{\psi}_{n,-}:= ^t\overline{\wh{\varphi}}_{n}$ is a representative of $^t\overline{\varphi}_{n}$ which converges in $\End(\Kb^{d+1})$ to 
\begin{align*}
\Psi_- := k_2^{-1}  \begin{pmatrix} \lambda_1^+ & & \\ & \ddots & \\ & & \lambda_{d+1}^+ \end{pmatrix} k_1^{-1}.
\end{align*}

Next let $m = \max\{ j : \lambda_j^+ \neq 0\}$ and $M = \min\{ j : \lambda_j^- \neq 0\}$. Then $m < M$ because $\varphi_n \rightarrow \infty$ in $\PGL_{d+1}(\Kb)$. 

Next let $e_1, \dots, e_{d+1}$ be the standard basis of $\Kb^{d+1}$. Then $\Phi_+$ maps any open set of $\Pb(\Kb^{d+1}) \setminus \ker \Phi$ onto an open set of $k_1 \Spanset\{e_1, \dots, e_m\}$. Moreover 
\begin{align*}
\Phi_+(z) = \lim_{n \rightarrow \infty} \varphi_{n}(z)
\end{align*}
for any $z \in \Pb(\Kb^{d+1}) \setminus \ker \Phi$ and since $\Aut(\Omega)$ acts properly on $\Omega$ if $q \in \Omega$ any limit point of $\varphi_n q$ is in $\partial \Omega$. Thus since $\Omega$ is open we see that $\partial \Omega$ contains an open subset of $k_1 \Spanset\{e_1, \dots, e_m\}$. The same argument applied to $\Psi_+$ implies that $\Omega^*$ contains an open subset of $k \Spanset\{e_M, \dots, e_{d+1}\}$. However, for $z_1 \in k_1 \Spanset\{e_1, \dots, e_m\}$ and $z_2 \in k_1 \Spanset\{e_M, \dots, e_{d+1}\}$ we have that $\ip{z_1, z_2}=0$. So if 
\begin{align*}
z_1 \in \partial \Omega \cap k_1 \Spanset\{e_1, \dots, e_m\}
\end{align*}
and 
\begin{align*}
z_2 \in \Omega^* \cap k_1 \Spanset\{e_M, \dots, e_{d+1}\}
\end{align*}
we see that 
\begin{align*}
\ker \ip{\cdot, z_2} = T_{z_1}^{\Kb} \partial \Omega.
\end{align*}
Thus $\dimension_{\Kb} k_1 \Spanset\{e_M, \dots, e_{d+1}\} = 1$ and so $M=d+1$. Applying this argument to $\Phi_-$ and $\Psi_-$ we see that $m=1$. 

Now $\operatorname{Im} \Phi_{\pm}=y^\pm$, $\operatorname{Im} \Psi_\pm = f^\pm$, and $\ip{y^\pm, f^\pm}=0$ for some $y^\pm, f^\pm \in \Pb(\Kb^{d+1})$. By the arguments above $y^\pm \in \partial \Omega$ and $f^\pm \in \Omega^*$. So $T_{y^\pm}^{\Kb} \partial \Omega = \ker f^{\pm}$. On the other hand, by construction, $\ker \Phi_\pm = \ker f^\mp$. So $\ker \Phi \cap \Omega = \emptyset$ and for all $q \in \Omega$ we have 
\begin{align*}
y^\pm = \Phi_{\pm}(q) = \lim_{n \rightarrow \infty} \varphi_{n}^{\pm 1} q.
\end{align*}
So $y^\pm = x^\pm$, $T_{x^\pm}^{\Cb} \partial \Omega = \ker f^{\pm}$, and $\ker \Phi = T_{x^-}^{\Cb} \partial \Omega$. This proves part (3).

Part (1), (2), and (4) follow from the proof of part (3). 

\end{proof}

\section{The structure of bi-proximal automorphisms}

Suppose $\varphi \in \PGL_{d+1}(\Kb)$ and $\wh{\varphi} \in \GL_{d+1}(\Kb)$ is a representative of $\varphi$ with $\det \wh{\varphi} = \pm 1$ (see Appendix~\ref{sec:quaternions} for the definition of $\det$ when $\Kb=\Hb$). Then let
\begin{align*}
\sigma_1(\varphi) \leq \sigma_2(\varphi) \leq \dots \leq \sigma_{d+1}(\varphi)
\end{align*}
be the absolute value of the eigenvalues (counted with multiplicity) of $\wh{\varphi}$. Since we are considering absolute values these numbers only depend on $\varphi$.

An element $\varphi \in \PGL_{d+1}(\Kb)$ is called \emph{proximal} if $\sigma_{d}(\varphi) < \sigma_{d+1}(\varphi)$ and is called \emph{bi-proximal} if $\varphi$ and $\varphi^{-1}$ are proximal. When $\varphi$ is bi-proximal let $x^+_{\varphi}$ and $x^-_{\varphi}$ be the eigenlines in $\Pb(\Kb^{d+1})$ corresponding to $\sigma_{d+1}(\varphi)$ and $\sigma_1(\varphi)$.  

\begin{proposition}\label{prop:bi_prox_str}
Suppose $\Omega \subset\Pb(\Kb^{d+1})$ is proper domain with $C^1$ boundary, $\varphi \in \Aut(\Omega)$, and $\sigma_{d+1}(\varphi) > \sigma_1(\varphi)$. Then $\varphi$ is bi-proximal. Moreover, 
\begin{enumerate}
\item $x_{\varphi}^+,x_{\varphi}^- \in \partial \Omega$, 
\item $T^{\Kb}_{x_{\varphi}^+} \partial \Omega \cap \partial \Omega = \{ x_{\varphi}^+\}$,
\item $T^{\Kb}_{x_{\varphi}^-} \partial \Omega \cap \partial \Omega = \{ x_{\varphi}^-\}$, and
\item if $U^+ \subset \overline{\Omega}$ is a neighborhood of $x^+_{\varphi}$ and $U^- \subset \overline{ \Omega}$ is a neighborhood of $x^-_{\varphi}$ then there exists $N>0$ such that for all $m > N$ we have
\begin{align*}
\varphi^m(\partial \Omega \setminus U^-) \subset U^+ \text{ and } \varphi^{-m}(\partial \Omega \setminus U^+) \subset U^-.
\end{align*}
\end{enumerate}
\end{proposition}

\begin{proof}
Since $\sigma_{d+1}(\varphi) > \sigma_1(\varphi)$, $\varphi^n \rightarrow \infty$ in $\PGL_{d+1}(\Kb)$.  So fixing $p \in \Omega$ we can find $n_k \rightarrow \infty$ so that 
\begin{align*}
\varphi^{n_k} p \rightarrow x^+ \text{ and } \varphi^{-n_k} p \rightarrow x^-
\end{align*}
for some $x^+, x^- \in \partial \Omega$. By Proposition~\ref{prop:limits} $\varphi^{n_k}$ converges in $\Pb(\End(\Kb^{d+1}))$ to an element $\Phi$ where $\operatorname{Im}(\Phi) = x^+$, $\ker \Phi = T_{x^-}^{\Kb} \partial \Omega$. Moreover, there exists $f^\pm \in \Omega^*$ with $\ker f^\pm = T^{\Cb}_{x^\pm} \partial \Omega$. 

By considering the Jordan block decomposition of $\varphi$ we see that $x^+$ is an eigenline of $\varphi$ with corresponding eigenvalue having absolute value $\sigma_{d+1}(\varphi)$ and $f^-$ is an eigenline of $^t\varphi$ with corresponding eigenvalue having absolute value $\sigma_{1}(\varphi)$. Applying this argument to $\varphi^{-1}$ implies that $x^-$ is an eigenline of $\varphi$ with corresponding eigenvalue having absolute value $\sigma_{1}(\varphi)$ and $f^+$ is an eigenline of $^t\varphi$ with corresponding eigenvalue having absolute value $\sigma_{d+1}(\varphi)$.

Now since $\sigma_1(\varphi) \neq \sigma_{d+1}(\varphi)$ we see that $f^+ \neq f^-$. Then $f^+(x^-) \neq 0$, for otherwise 
\begin{align*}
\ker f^- = T_{x^+}^{\Kb} \partial \Omega = \ker f^+
\end{align*}
which is impossible. Similarly, $f^-(x^+) \neq 0$. 

Now $\varphi$ preserves the subspaces $x^+$, $x^-$, and $\ker f^+ \cap \ker f^-$. So if $v_1, \dots, v_{d+1}$ is a basis of $\Kb^{d+1}$ with $\Kb v_1 = x^+$, $\Kb v_2 = x^-$, and $\ker f^+ \cap \ker f^- = \Spanset_{\Kb}(v_3, \dots, v_{d+1})$ then with respect to this basis $\varphi$ is represented by a matrix of the form 
\begin{align*}
\begin{pmatrix} \lambda^+ & & \\ & \lambda^- & \\ & & A \end{pmatrix} \in \GL_{d+1}(\Kb).
\end{align*}
Since $\operatorname{Im}(\Phi) = x^+$ we see that $\norm{A} < \abs{\lambda^+}$ and applying this argument to $\varphi^{-1}$ implies that  $\norm{A^{-1}} < \abs{\lambda^-}^{-1}$. Thus $\varphi$ is bi-proximal and $x^\pm = x^\pm_{\varphi}$.

We next claim that $ \partial \Omega \cap T_{x^+}^{\Kb} \partial \Omega =\{x^+\}$. Suppose that $z \in  \partial \Omega \cap T_{x^+}^{\Kb} \partial \Omega$ then either $z = x^+$ or $z = [z_1 : 0 : z_2 : \dots : z_{d}]$ and $z_j \neq 0$ for some $2 \leq j \leq d$. In the latter case, there exist $m_i \rightarrow \infty$ such that $\varphi^{-m_i} z \rightarrow w$ and $w = [0:0 : w_2 : \dots : w_d]$. But then $w \in  \partial \Omega \cap T_{x^+}^{\Kb} \partial \Omega \cap T_{x^-}^{\Kb} \partial \Omega$ which is impossible since $\partial \Omega$ is $C^1$. So we have a contradiction and so $z = x^+$. Applying this argument to $\varphi^{-1}$ shows that $ \partial \Omega \cap T_{x^-}^{\Kb} \partial \Omega =\{x^-\}$. 

Finally part (4) follows part (4) of Proposition~\ref{prop:limits}. 
\end{proof}

\section{Finding bi-proximal elements}

\begin{theorem}\label{thm:bi_prox_exist}
Suppose $\Omega \subset\Pb(\Kb^{d+1})$ is proper domain with $C^1$ boundary. If there exists $x,y \in \Lc(\Omega)$ such that $T_x^{\Kb} \partial \Omega \neq T_y^{\Kb} \partial \Omega$ then $\Aut(\Omega)$ contains a bi-proximal element. 
\end{theorem}

\begin{lemma}\label{lem:duality}
Suppose $\Omega \subset\Pb(\Kb^{d+1})$ is proper domain with $C^1$ boundary, $\varphi_n \in \Aut(\Omega)$, and 
\begin{align*}
\varphi_n p \rightarrow x^+ \text{ and } \varphi_n^{-1} p \rightarrow x^-
\end{align*}
where  $p \in \Omega$ and $x^+, x^- \in \partial \Omega$. If $T_{x^+}^{\Kb} \partial \Omega \neq T_{x^-}^{\Kb} \partial \Omega$ then $\varphi_n$ is bi-proximal for $n$ large enough. Moreover, $x_{\varphi_n}^+ \rightarrow x^+$ and $x_{\varphi_n}^- \rightarrow x^-$.
\end{lemma}

Given two points $x,y \in \Pb(\Kb^{d+1})$ let $L(x,y)$ be the projective line containing $x$ and $y$. 

\begin{proof}
We first claim that for $n$ large enough $\varphi_n$ has fixed points $x_n^+, x_n^- \in \overline{\Omega}$. Fix compact neighborhoods $U^{\pm}$ of $x^\pm$ with the following properties:
\begin{enumerate}
\item $U^\pm \cap T_{x^\mp}^{\Kb} \partial \Omega = \emptyset$,
\item $U^\pm \cap \overline{\Omega}$ is topologically a closed ball, 
\item there exist a compact set $K \subset \Omega$ so that if $y^+ \in U^+$ and $y^- \in U^-$ then $L(y^+, y^-) \cap K \neq \emptyset$.
\end{enumerate}
Since $x^\pm \notin T_{x^\mp}^{\Cb} \partial \Omega$ part (1) holds for small enough neighborhoods. Since $\partial \Omega$ is a $C^1$ hypersurface it is always possible to shrink a neighborhood so that part (2) holds. Finally since the line $L(x^+, x^-)$ is transverse to $\partial \Omega$ at $x^+$ and $x^-$, part (3) holds for small enough neighborhoods. 

Now by Proposition~\ref{prop:limits} there exists $N \geq 0$ so that 
\begin{align*}
\varphi_n(U^\pm \cap  \overline{\Omega}) \subset U^\pm \cap \overline{\Omega}
\end{align*}
for all $n \geq N$. So by the Bouwer fixed point theorem, for $n$ large enough  $\varphi_n$ has a fixed point $x^\pm_n \in U^\pm \cap \overline{\Omega}$. Now fix points $k_n \in K \cap L(x_n^+, x_n^-)$. Since $K \subset \Omega$ is compact and $\varphi_n^{\pm 1} q \rightarrow x^\pm$ for all $q \in \Omega$ we see that 
\begin{align*}
\varphi_n k_n \rightarrow x^+ \text{ and } \varphi_n^{-1} k_n \rightarrow x^-.
\end{align*}
So for large $n$ the ratio of the absolute value of the eigenvalues of $\varphi_n$ corresponding to the lines $x_n^+$ and $x_n^-$ must be different. So for large $n$, $\sigma_{d+1}(\varphi_n) > \sigma_1(\varphi_n)$. Thus $\varphi_n$ is bi-proximal by Proposition~\ref{prop:bi_prox_str}. Then by part (4) of Proposition~\ref{prop:bi_prox_str} we see that $x_n^{\pm} = x_{\varphi_n}^{\pm}$. 

Finally we can choose $U^+$ and $U^-$ to be arbitrary small neighborhoods of $x^+$ and $x^-$ which implies that $x_{\varphi_n}^+ \rightarrow x^+$ and $x_{\varphi_n}^- \rightarrow x^-$.
\end{proof}

\begin{lemma}\label{lem:generic_endpoints}
Suppose $\Omega \subset\Pb(\Kb^{d+1})$ is proper domain with $C^1$ boundary, $\varphi_n, \phi_m \in \Aut(\Omega)$,  
\begin{align*}
\varphi_n p \rightarrow x^+ , \varphi_n^{-1} p \rightarrow x^-, \phi_m p \rightarrow y^+, \text{ and } \phi_m^{-1} p \rightarrow y^-
\end{align*}
where $p \in \Omega$ and $x^+, x^-, y^+, y^- \in \partial \Omega$. If 
\begin{align*}
\left\{ T_{x^+}^{\Kb} \partial \Omega, T_{x^-}^{\Kb} \partial \Omega\right\} \cap \left\{ T_{y^+}^{\Kb} \partial \Omega, T_{y^-}^{\Kb} \partial \Omega\right\} = \emptyset
\end{align*}
then $\gamma_k : = \varphi_{k} \phi_{k}^{-1}$ is bi-proximal for $k$ large enough. Moreover
\begin{align*}
\gamma_k p \rightarrow x^+ \text{ and } \gamma_k^{-1} p \rightarrow y^+.
\end{align*}
\end{lemma}

\begin{proof}
Fix compact neighborhoods $U^{\pm}$ of $x^\pm$ and $V^{\pm}$ of $y^\pm$ so that 
\begin{align*}
(U^+ \cup U^-) \cap \left(  T_{y^+}^{\Kb} \partial \Omega \cup T_{y^-}^{\Kb} \partial \Omega \right)= \emptyset
\end{align*}
and 
\begin{align*}
(V^+ \cup V^-) \cap \left(  T_{x^+}^{\Kb} \partial \Omega \cup T_{x^-}^{\Kb} \partial \Omega \right)= \emptyset.
\end{align*}
Now by Proposition~\ref{prop:limits} there exists $N \geq 0$ so that $\varphi_n^{-1} p \in U^-$ and $\varphi_n (V^+ \cup V^-) \subset U^+$ for all $n \geq N$. Likewise there exists $M \geq 0$ so that $\phi_m^{-1} p \in V^-$ and $\phi_m(U^+ \cup U^-) \subset V^+$ for all $m \geq M$. Then if $k \geq \max\{M,N\}$ and $\gamma_k := \varphi_k \phi_k^{-1}$ we see that $\gamma_k p \in U^+$ and $\gamma_k^{-1} p \in V^+$. 

Since $U^+$ and $V^+$ can be choosen to be arbitrary small neighborhoods of $x^+$ and  $y^+$ respectively, we see that 
\begin{align*}
\gamma_{k} p \rightarrow x^+ \text{ and } \gamma_{k}^{-1} p \rightarrow y^+.
\end{align*}
Finally since
\begin{align*}
 T_{x^+}^{\Kb} \partial \Omega \neq T_{y^+}^{\Kb} \partial \Omega
\end{align*}
Lemma~\ref{lem:duality} implies that $\gamma_{k}$ is bi-proximal for large $k$. 
\end{proof}

\begin{proof}[Proof of Theorem~\ref{thm:bi_prox_exist}]
Fix sequences $\varphi_n, \phi_m \in \Aut(\Omega)$ so that 
\begin{align*}
\varphi_n p \rightarrow x \text{ and } \phi_m p \rightarrow y
\end{align*}
for some $p \in \Omega$. By passing to a subsequence we may suppose that 
\begin{align*}
\varphi_n^{-1} p \rightarrow x^- \text{ and } \phi_m^{-1} p \rightarrow y^-.
\end{align*}
If $T_x^{\Kb} \partial \Omega \neq T_{x^-}^{\Kb} \partial \Omega$ then $\varphi_n$ is bi-proximal for large $n$ by Lemma~\ref{lem:duality}. Likewise if $T_y^{\Kb} \partial \Omega \neq T_{y^-}^{\Kb} \partial \Omega$ then $\phi_m$ is bi-proximal for large $m$ by Lemma~\ref{lem:duality}. 

So suppose that $T_x^{\Kb} \partial \Omega = T_{x^-}^{\Kb} \partial \Omega$ and $T_y^{\Kb} \partial \Omega = T_{y^-}^{\Kb} \partial \Omega$. Then 
\begin{align*}
\left\{ T_{x}^{\Kb} \partial \Omega, T_{x^-}^{\Kb} \partial \Omega\right\} \cap \left\{ T_{y}^{\Kb} \partial \Omega, T_{y^-}^{\Kb} \partial \Omega\right\} = \emptyset
\end{align*}
and so $\varphi_{k} \phi_{k}^{-1}$ is bi-proximal for large $k$ by Lemma~\ref{lem:generic_endpoints}. 
\end{proof}

\section{Rescaling with bi-proximal elements}

\begin{definition} If $\Kb$ is either $\Rb$, $\Cb$, or $\Hb$ let $\Kb_{P}$ be the \emph{purely imaginary numbers} in $\Kb$, that is $\Rb_P = (0)$, $\Cb_P = i\Rb$, and $\Hb_P = i\Rb + j\Rb + k\Rb$. 
\end{definition}

Suppose $\Omega$ is a proper domain with $C^1$ boundary. If $\varphi \in \Aut(\Omega)$ is bi-proximal, then we have the following standard form. First let $H^{\pm}$ be the $\Kb$-tangent hyperplane at $x^{\pm}_\varphi$. Then pick coordinates such that
\begin{enumerate}
\item $x^+_\varphi = [1:0:\dots:0]$,
\item $x^-_\varphi = [0:1:0: \dots :0]$,
\item $H^+ \cap H^- = \{ [0:0:z_2:\dots:z_d] : z_2,\dots, z_d \in \Kb\}$.
\end{enumerate}
With respect to these coordinates, $\varphi$ is represented by a matrix of the form
\begin{align*}
\begin{pmatrix}
\lambda^+ &  & \\
 & \lambda^- &  \\
 & & A
\end{pmatrix} \in \GL_{d+1}(\Kb)
\end{align*}
where $A$ is a $(d-1)$-by-$(d-1)$ matrix. Since 
\begin{align*}
H^-=\{[0:z_1:\dots:z_d]: z_1 \dots, z_d \in \Kb\}
\end{align*}
and $\Omega \cap H^- = \emptyset$ we see that $\Omega$ is contained in the affine chart 
\begin{align*}
\Kb^d = \{ [1: z_1: \dots : z_d ] : z_1, \dots, z_d \in \Kb\}.
\end{align*}
In this affine chart $x^+_\varphi$ corresponds to $0$ and $T^{\Kb}_0 \partial \Omega = \{0\} \times \Kb^{d-1}$. Then by a projective transformation we may assume that
\begin{enumerate}
\setcounter{enumi}{3}
\item $T_0 \partial \Omega =  \Kb_P \times \Kb^{d-1}$
\end{enumerate}
Since $\partial \Omega$ is $C^1$ there exists open neighborhoods $V \subset \Kb_P$ of $0$, $W \subset \Rb$ of $0$, an open neighborhood $U \subset \Kb^{d-1}$ of $0$, and a $C^1$ function $F:V \times U \rightarrow W$ such that if $\Oc = (V+W) \times U$ then 
\begin{enumerate}
\setcounter{enumi}{4}
\item $\partial \Omega \cap \Oc = \{ (z_1, \dots, z_d) \in \Oc : \Real(z_1) = F(z_1 - \Real(z_1), z_2, \dots, z_d)  \}$.
\end{enumerate}
By another projective transformation we can assume 
\begin{enumerate}
\setcounter{enumi}{5}
\item $\Omega \cap \Oc = \{ (z_1, \dots, z_d) \in \Oc : \Real(z_1) > F(z_1 - \Real(z_1), z_2, \dots, z_d)  \}$.
\end{enumerate}

\begin{theorem}
\label{thm:blow_up}
With the choice of coordinates above, the function $F$ extends to $\Kb_P \times \Kb^{d-1}$ and 
\begin{align*}
\Omega = \{ (z_1, \dots, z_d) \in \Kb^d : \Real(z_1) >F(0, z_2, \dots, z_d) \}.
\end{align*}
Moreover if $\Kb$ is either $\Cb$ or $\Hb$ then for 
\begin{align*}
h = \begin{pmatrix} a & b \\ c & d \end{pmatrix} \in \SL_2(\Kb)
\end{align*}
with $[h] \in \Aut_0(\{ z \in \Kb : \Real(z) > 0\})$ the projective transformation defined by 
\begin{align*}
\psi_h \cdot [z_1, \dots, z_d] = [ az_1 + bz_2 : cz_1 + d z_2 : z_3 : \dots : z_{d+1}]
\end{align*}
is in $\Aut_0(\Omega)$.
\end{theorem}

\begin{remark} A special case of the above Theorem was established in~\cite[Theorem 6.1]{Z2013}. Namely when $\Kb = \Cb$ and in addition $\Omega$ is a $\Cb$-convex set. 
\end{remark}

\begin{proof}
We can assume $\Oc$ is bounded. Then by Proposition~\ref{prop:limits} we can replace $\varphi$ with a power of $\varphi$ so that $\varphi(\Oc) \subset \Oc$.

We first claim that $F(x,z) = F(0,z)$ for $(x,z) \in V \times U$. Notice that with our choice of coordinates $\varphi$ acts by
\begin{align*}
\varphi \cdot (z_1, \vec{z}) = \left( \frac{\lambda^-z_1}{\lambda^+}, \frac{A\vec{z}}{\lambda^+}\right)
\end{align*}
where $\lambda^{\pm}$ and $A$ are as above. Since $\varphi$ is bi-proximal 
\begin{align*}
\left(\frac{A}{\lambda^+}\right)^n \rightarrow 0.
\end{align*}
Since $\varphi$ preserves $T_0\partial \Omega = \Kb_P \times \Kb^{d-1}$ we see that $\lambda^-/\lambda^+ \in \Rb$. Since $x^+_{\varphi}$ is an attracting fixed point we have $\lambda^-/\lambda^+ \in (-1,1)$. Finally since
\begin{align*}
\varphi \cdot (x+F(x,z), z) = \left( \frac{\lambda^-}{\lambda^+}x+\frac{\lambda^-}{\lambda^+}F\left(x,z\right), \frac{A}{\lambda^+}z\right)
\end{align*}
and $\varphi(\Oc) \subset \Oc$ we see that 
\begin{align*}
F\left(\frac{\lambda^-}{\lambda^+}x,\frac{A}{\lambda^+}z\right)=\frac{\lambda^-}{\lambda^+}F(x,z).
\end{align*}
Differentiating $F$ in the $x$ direction yields
\begin{align*}
(\nabla_x F)(x,z) = (\nabla_x F)\left(\frac{\lambda^-}{\lambda^+}x, \frac{A}{\lambda^+}z\right)
\end{align*}
and repeated applications of the above formula shows
\begin{align*}
(\nabla_x) F(x,z) = (\nabla_x F)\left(\left(\frac{\lambda^-}{\lambda^+}\right)^nx, \left(\frac{A}{\lambda^+}\right)^nz\right)
\end{align*}
for all $n>0$. Taking the limit as $n$ goes to infinity proves that $(\nabla_x F)(x,z) = (\nabla_x F)(0,0)$. Since $(\nabla_x F)(0,0)=0$ we then see that $F(x,z) = F(0,z)$ for all $(x,z) \in V \times U$. 

Now for $(x, z) \in \Kb_P \times \Kb^{d-1}$ there exists $N >0$ so that $\varphi^N \cdot (x,z) \in V \times U$. Then we define 
\begin{align*}
F(x,z) := \left(\frac{\lambda^+}{\lambda^-}\right)^NF\left(\left(\frac{\lambda^-}{\lambda^+}\right)^Nx,\left(\frac{A}{\lambda^+}\right)^Nz \right).
\end{align*}
Notice that this definition does not depend on the choice of $N$, that is if $\varphi^M \cdot (x,z) \in V \times U$ then 
\begin{align*}
\left(\frac{\lambda^+}{\lambda^-}\right)^MF\left(\left(\frac{\lambda^-}{\lambda^+}\right)^Mx,\left(\frac{A}{\lambda^+}\right)^Mz \right)= \left(\frac{\lambda^+}{\lambda^-}\right)^NF\left(\left(\frac{\lambda^-}{\lambda^+}\right)^Nx,\left(\frac{A}{\lambda^+}\right)^Nz \right).
\end{align*}
So we see that $F$ extends to a  function defined on $\Kb_P \times \Kb^{d-1}$. Moreover, this function is clearly $C^1$. With this extension 
\begin{align*}
\cup_{n \in \Nb} \varphi^{-n} (\Oc \cap \Omega) = \{ (z_1, \dots, z_d) \in \Kb^d : \Real(z_1) >F(0, z_2, \dots, z_d) \}
\end{align*}
and $\cup_{n \in \Nb} \varphi^{-n} (\Oc \cap \Omega) = \Omega$ by Proposition~\ref{prop:bi_prox_str}. This proves the first part of the Theorem. 

Now assume that $\Kb$ is either $\Cb$ or $\Hb$. Then for $w \in \Kb_P$ define the projective map $u_w$ by $u_w \cdot (z_1, \dots, z_d) = (z_1+w, z_2, \dots, z_d)$. Since $F(x,z) = F(0,z)$, we see that $u_w \in \Aut_0(\Omega)$ for all $w \in \Kb_P$. Also $u_w$ corresponds to the matrix 
\begin{align*}
\begin{pmatrix}
1 & 0 \\ 
w & 1 
\end{pmatrix}
\end{align*}
in the action of $\SL_2(\Kb)$ defined in the statement of the theorem.

The same argument starting with $\varphi^{-1}$ instead of $\varphi$ (that is viewing $\Omega$ as a subset of the affine chart $\{ [z_1 : 1 : z_2 : \dots : z_d]\}$) shows that $\Aut_0(\Omega)$ contains the group of automorphisms corresponding to the matrices
\begin{align*}
\left\{ \begin{pmatrix}
1 & w \\ 
0 & 1 
\end{pmatrix} : w \in \Kb_P \right\}
\end{align*}
in the action of $\SL_2(\Kb)$ defined in the statement of the theorem.  

Finally these two one-parameter subgroups generate all of $\Aut_0( \{ z \in \Kb : \Real(z) > 0\} )$ (see Proposition~\ref{prop:auto_halfspace}) and thus the second part of the theorem is proven. 
\end{proof}

We end the section with three corollaries of Theorem~\ref{thm:blow_up}. If we consider the matrices 
\begin{align*}
h = 
\begin{pmatrix} 
e^t & 0 \\
0 & e^{-t}
\end{pmatrix}
\end{align*}
in the statement of Theorem~\ref{thm:blow_up} then we have the following:

\begin{corollary}
\label{cor:special_bi}
Suppose  $\Kb$ is either $\Cb$ or $\Hb$ and $\Omega$ is a proper domain with $C^1$ boundary. If $\varphi \in \Aut(\Omega)$ is bi-proximal, then there exists a one-parameter subgroup $\psi_t \in \SL_{d+1}(\Kb)$ of bi-proximal elements such that $[\psi_t] \in \Aut_0(\Omega)$ and
\begin{enumerate}
\item $(\psi_t)|_{x^+_{\varphi}} = e^t\Id|_{x^+_{\varphi}} $,
\item $(\psi_t)|_{x^-_{\varphi}} = e^{-t}\Id|_{x^-_{\varphi}} $,
\item $(\psi_t)|_{H^+ \cap H^-} = \Id|_{H^+ \cap H^-}$ where $H^{\pm} = T_{x^\pm_{\varphi}}^{\Kb} \partial \Omega$.
\end{enumerate}
\end{corollary}

Proposition~\ref{prop:auto_halfspace} and Theorem~\ref{thm:blow_up} also imply the following:

\begin{corollary}
\label{cor:bi_3}
Suppose $\Kb$ is either $\Cb$ or $\Hb$, $\Omega$ is a proper domain with $C^1$ boundary, $\varphi \in \Aut(\Omega)$ is bi-proximal, and $L$ is the complex projective line containing $x^+_{\varphi}$ and $x^-_{\varphi}$. Then for all $x,y \in L \cap \partial\Omega$ there exists $\varphi_{xy} \in \Aut_0(\Omega)$ such that $\varphi_{xy}(x)=y$.
\end{corollary}

Finally Theorem~\ref{thm:blow_up} also implies the following:

\begin{corollary}
\label{cor:different_endpoints}
Suppose $\Kb$ is either $\Cb$ or $\Hb$, $\Omega$ is a proper domain with $C^1$ boundary, and $\Aut(\Omega)$ contains a bi-proximal element. If $x_1,\dots, x_n \in \partial \Omega$ then there exists a bi-proximal element $\varphi \in \Aut(\Omega)$ so that 
\begin{align*}
\{x_1,\dots, x_n\} \cap\{x^+_{\varphi}, x^-_{\varphi}\} = \emptyset.
\end{align*}
\end{corollary}

\section{Proof of Theorem~\ref{thm:C2}}

\subsection{The real case} First suppose that $\Kb = \Rb$. Then using Theorem~\ref{thm:bi_prox_exist} and Theorem~\ref{thm:blow_up} we can find an bi-proximal element $\varphi \in \Aut(\Omega)$ and make a change of variable so that the affine chart $\Rb^d=\{ [1:x_1: \dots : x_d] : x_1, \dots, x_d \in \Rb\}$ contains $\Omega$, $0 \in \partial \Omega$, $T_0 \partial \Omega = \{0\} \times \Rb^{d-1}$,  
\begin{align*}
\Omega = \{ (x_1, \dots, x_d) \in \Rb^d : x_1 > F(x_2, \dots, x_d) \}
\end{align*}
for some $C^2$ function $F: \Rb^{d-1} \rightarrow \Rb$, and $\varphi$ has attracting and repelling fixed points $x^+ = [1:0:\dots :0]$ and $x^-=[0:1:0:\dots:0]$. Notice that $F(0) =0$ and $F(x) > 0$ for all $x \in \Rb^{d-1} \setminus \{0\}$. 

Now with respect to these coordinates $\varphi$ is represented by a  matrix of the form 
\begin{align*}
\begin{pmatrix} 
\lambda & & \\ 
 & \frac{1}{\lambda} & \\
& & A \end{pmatrix} \in \GL_{d+1}(\Rb)
\end{align*}
with $\lambda > 1$. And so
\begin{align*}
\lambda^{2n} F\left(\frac{1}{\lambda^n} A^nx \right) = F(x)
\end{align*}
for all $x \in \Rb^{d-1}$ and $n \in \Nb$. 

We first claim that up to a change of coordinates $A \in O(d-1)$. Since $F$ is $C^2$ there exists $C >0$ so that $F(x) \leq C\norm{x}^2$ for all $x$ sufficiently close to $0$. Thus for $n$ large enough 
\begin{align*}
F(x) \leq \lambda^{2n}\norm{ \frac{A^n}{\lambda^n} x}^2 = \norm{A^n x}^2.
\end{align*}
Since $F$ is positive on $\Rb^{d-1} \setminus\{0\}$ this implies that 
\begin{align*}
\inf_{n \in  \Nb} \inf_ {\norm{x}=1} \norm{A^n x} > 0.
\end{align*}
Applying the same argument to $\varphi^{-1}$ shows that 
\begin{align*}
\inf_{n \in  \Nb} \inf_ {\norm{x}=1} \norm{A^{-n} x} > 0.
\end{align*}
Thus $\norm{A^n}$ and $\norm{A^{-n}}$ are both bounded. Thus $\{ A^n : n \in \Zb\} \leq \GL_{d-1}(\Rb)$ is a bounded group. Thus, up to a change of coordinates, $A \in O(d-1)$. 

Now we can fix $n_k \rightarrow \infty$ so that $A^{n_k} \rightarrow \Id_{d-1}$. Then for $x \in \Rb^{d-1}$
\begin{align*}
F(x) = \lim_{k \rightarrow \infty} \lambda^{2n_k} F\left(\frac{1}{\lambda^{n_k}} A^{n_k}x \right) =  \frac{1}{2} \mathrm{Hess}(F)_{0}(x,x)
\end{align*}
since $F$ is $C^2$. Since $F(x) > 0$ for all non-zero $x$ we then see that $\mathrm{Hess}(F)_{0}$ is positive definite and hence up to a change of basis we see that 
\begin{align*}
F(x_2, \dots, x_d) = \frac{1}{2} \sum_{i=2}^{d} \abs{x_i}^2.
\end{align*}

\subsection{The complex and quaternionic case}
First, using Theorem~\ref{thm:bi_prox_exist} and Corollary~\ref{cor:special_bi}, we can change coordinates so that the affine chart 
\begin{align*}
\Kb^d=\{ [1:z_1 : \dots : z_d] : z_1, \dots, z_d \in \Kb\}
\end{align*} 
contains $\Omega$, $0 \in \partial \Omega$, $T_0 \partial \Omega = \Kb_P \times \Kb^{d-1}$, and
\begin{align*}
\Omega = \{ (z_1, \dots, z_d) \in \Cb^d : \Real(z_1) > F(z_2, \dots, z_d) \}
\end{align*}
for some $C^2$ function $F: \Cb^{d-1} \rightarrow \Rb$. Notice that $F(0) =0$ and $F(z) > 0$ for all $z \in \Kb^{d-1} \setminus \{0\}$. 

We also can assume that $\Aut_0(\Omega)$ contains the transformation
\begin{align*}
[z_1, \dots, z_d] \rightarrow [ az_1 + bz_2 : cz_1 + d z_2 : z_3 : \dots : z_{d+1}]
\end{align*}
when 
\begin{align*}
h = \begin{pmatrix} a & b \\ c & d \end{pmatrix} \in\SL_2(\Kb)
\end{align*}
and $[h] \in \Aut_0( \{ z \in \Kb : \Real(z) > 0\} )$.

We claim that
\begin{align}
\label{eq:goal}
F(z) = \frac{1}{\abs{w}^2} F(wz)
\end{align}
for $w \in \Kb \setminus \{0\}$. First since the transformation 
\begin{align*}
a_t \cdot [z_1, \dots, z_d] = [ e^{-t} z_1 : e^{t} z_2 : z_3 : \dots : z_{d+1}]
\end{align*}
acts on the affine chart $\Kb^d$ by 
\begin{align*}
a_t \cdot (z_1, \dots, z_d) = (e^{2t} z_1, e^t z_2, \dots, e^t z_d)
\end{align*}
we see that 
\begin{align}
\label{eq:goal_1}
F(z) = \frac{1}{e^{2t}} F(e^tz).
\end{align}
Next if $w \in \Kb_P$ the transformation 
\begin{align*}
u_w \cdot [z_1, \dots, z_d] = [ z_1+wz_2 : z_2 : z_3 : \dots : z_{d+1}]
\end{align*}
is in $\Aut(\Omega)$ and acts on the affine chart $\Kb^d$ by 
\begin{align*}
u_w \cdot (z_1, z_2, \dots, z_d) = \left( \frac{z_1}{1+wz_1}, \frac{z_2}{1+wz_1}, \dots, \frac{z_d}{1+wz_1}\right).
\end{align*}
Notice that
\begin{align*}
\Real \left(  \frac{z_1}{1+wz_1} \right)  = \frac{1}{\abs{1+wz_1}^2} \Real \left( z_1(1-\overline{z}_1w )\right) = \frac{\Real(z_1)}{\abs{1+wz_1}^2}
\end{align*}
so if we apply $u_{w/F(z)}$ to the point $(F(z), z) \in \partial \Omega$ we see that
\begin{align}
\label{eq:goal_2}
F(z) = \abs{1+w}^2 F\left( \frac{z}{1+w}\right)
\end{align}
for all $w \in \Kb_P$. Combining Equations~\ref{eq:goal_1} and~\ref{eq:goal_2} we see that Equation~\ref{eq:goal} holds for all $w \in \Kb$ with $\Real(w) > 0$. On the other hand, any $w \in \Kb \setminus \{0\}$ can be written as $z=w_1w_2$ where $\Real(w_1), \Real(w_2) > 0$. So Equation~\ref{eq:goal} holds for all $w \in \Kb \setminus \{0\}$.

Now since $F$ is $C^2$ and $F(z) = e^{2t} F(e^{-t}z)$ we see that 
\begin{align*}
F(z) = \frac{1}{2} \mathrm{Hess}(F)_{0}  ( z, z)
\end{align*}
for all $z \in \Kb^{d-1}$. Since $T_0^{\Cb} \partial \Omega \cap \partial \Omega = \{0\}$ and $\{0\} \times \Cb^{d-1}$ we have $F(z) >0$ for all $z \in \Cb^{d-1}$ and so the Hessian of $F$ is positive definite. 

Now let $r = \dimension_{\Rb} \Kb$ and identify $\Kb^{d-1}$ with $\Rb^{r(d-1)}$ in the obvious way. For $w \in \Kb$ let $M(w) \in \GL_{r(d-1)}(\Rb)$ denote the action by scalar multiplication by $w$. Notice that $^tM(w) = M(\overline{w})$. Now under this identification there exists a matrix $A \in \GL_{r(d-1)}(\Rb)$ so that 
 \begin{align*}
 \mathrm{Hess}(F)_{0}  ( z_1, z_2) = ^tz_1 A z_2.
 \end{align*}
Since $F(wz) = \abs{w}^2 F(z)$ for $w \in \Kb$ we see that 
\begin{align*}
M(\overline{w}) A M(w) = \abs{w}^2 A.
\end{align*}
So 
\begin{align*}
A M(w) = \abs{w}^2M(\overline{w})^{-1} A.
\end{align*}
But $\abs{w}^2M(\overline{w})^{-1} = M(w)$. Thus $A$ is $\Kb$-linear. Hence $A$ can be viewed as a matrix in $\GL_{d-1}(\Kb)$. Now since $^tA = A$ as a matrix in $\GL_{r(d-1)}(\Rb)$ we see that $^t\overline{A} = A$ as a matrix in $\GL_{d-1}(\Kb)$. Moreover, $A$ is positive semidefinite. Thus there exists $g \in \GL_{d-1}(\Kb)$ so that 
\begin{align*}
^t\overline{g}Ag = \Id_{d-1}.
\end{align*}
Thus, up to a change of coordinates, 
\begin{align*}
F(z_2, \dots, z_d) = \sum_{i=2}^d \abs{z_i}^2
\end{align*}
and 
\begin{align*}
\Omega = \left\{ (z_1, \dots, z_d) \in \Kb^d : \Real(z_1) > \sum_{i=2}^d \abs{z_i}^2 \right\}.
\end{align*}

\section{The structure of the limit set}

\begin{proposition}\label{prop:limit_set}
Suppose $\Kb$ is either $\Cb$ or $\Hb$ and $\Omega \subset\Pb(\Kb^{d+1})$ is a proper domain with $C^1$ boundary. If there exists $x,y \in \Lc(\Omega)$ such that $T_x^{\Kb} \partial \Omega \neq T_y^{\Kb} \partial \Omega$ then the limit set $\Lc(\Omega) \subset \Pb(\Kb^{d+1})$ is a closed $C^\infty$ submanifold of $\Pb(\Kb^{d+1})$ and $\Aut_0(\Omega)$ acts transitively on $\Lc(\Omega)$.
\end{proposition}

The fact that $\Lc(\Omega)$ is a $C^\infty$ submanifold of $\Pb(\Kb^{d+1})$ will follow from a general fact about the orbits of Lie groups:

\begin{lemma}\label{lem:smth_orbits}
Suppose $G$ is a connected Lie group acting smoothly on a smooth manifold $M$. Then an orbit $G\cdot m$ is an embedded smooth submanifold of $M$ if and only if $G \cdot m$ is locally closed in $M$. 
\end{lemma}

Here smooth mean $C^\infty$ and for a proof see~\cite[Theorem 15.3.7]{tD2008}. 

\begin{lemma}\label{lem:att_fix_closed}
Suppose $\Kb$ is either $\Cb$ or $\Hb$ and $\Omega \subset\Pb(\Kb^{d+1})$ is a proper domain with $C^1$ boundary. If $x^+,x^- \in \partial \Omega$, $T_{x^+}^{\Kb} \partial \Omega \neq T_{x^-}^{\Kb} \partial \Omega$, and there exists bi-proximal elements $\varphi_n \in \Aut(\Omega)$ so that 
\begin{align*}
x^+_{\varphi_n} \rightarrow x^+ \text{ and } x^-_{\varphi_n} \rightarrow x^-
\end{align*}
then there exists $\varphi \in \Aut(\Omega)$ bi-proximal so that $x^+ = x^+_{\varphi}$ and $x^- = x^-_{\varphi}$. 
\end{lemma}

\begin{proof}
Let $H_n^{\pm} = T_{x_{\varphi_n}^\pm}^{\Kb} \partial \Omega$ and $H^\pm = T_{x^\pm}^{\Kb} \partial \Omega$. Since $H^+ \neq H^-$, $H_n^+ \cap H_n^- \rightarrow H^+ \cap H^-$ in the space of $(d-1)$-planes in $\Kb^{d+1}$. By Corollary~\ref{cor:special_bi}  we can assume that $\varphi_n \in \SL_{d+1}(\Kb)$ and
\begin{enumerate}
\item $(\varphi_n)|_{x^+_{\varphi_n}} = 2 \Id|_{x^+_{\varphi_n}} $,
\item $(\varphi_n)|_{x^-_{\varphi_n}} = \frac{1}{2}\Id|_{x^-_{\varphi_n}} $,
\item $(\varphi_n)|_{H_n^+ \cap H_n^-} = \Id|_{H_n^+ \cap H_n^-}$.
\end{enumerate}
Since $H_n^+ \cap H_n^- \rightarrow H^+ \cap H^-$, $\varphi_n$ converges to  $\varphi \in \SL_{d+1}(\Kb)$ where 
\begin{enumerate}
\item $(\varphi)|_{x^+} = 2 \Id|_{x^+} $,
\item $(\varphi)|_{x^-} = \frac{1}{2}\Id|_{x^-} $,
\item $(\varphi)|_{H^+ \cap H^-} = \Id|_{H^+ \cap H^-}$.
\end{enumerate}
Then since $\Aut(\Omega)$ is closed $[\varphi] \in \Aut(\Omega)$. 
\end{proof}

\begin{lemma}
Suppose $\Kb$ is either $\Cb$ or $\Hb$, $\Omega \subset\Pb(\Kb^{d+1})$ is a proper domain with $C^1$ boundary, and there exists $x,y \in \Lc(\Omega)$ such that $T_x^{\Kb} \partial \Omega \neq T_y^{\Kb} \partial \Omega$. Then for any $z \in \Lc(\Omega)$ 
\begin{align*}
T_z^{\Kb} \partial \Omega \cap \partial \Omega = \{z\}.
\end{align*}
\end{lemma}

\begin{proof}
There exists $\phi_m \in \Aut(\Omega)$ and $p \in \Omega$ so that $\phi_m p \rightarrow z$. By passing to a subsequence we can assume that $\phi_m^{-1} p \rightarrow z^-$. Now by Theorem~\ref{thm:bi_prox_exist} and Corollary~\ref{cor:different_endpoints} there exists $\gamma \in \Aut(\Omega)$ bi-proximal so that $\{ z, z^-\} \cap \{ x^+_{\gamma}, x^-_{\gamma}\} = \emptyset$. Proposition~\ref{prop:bi_prox_str} implies that
\begin{align*}
T_{x^\pm_{\gamma}}^{\Kb} \partial \Omega \cap \partial \Omega= \{ x^{\pm}_{\gamma} \}
\end{align*}
and so 
\begin{align*}
\left\{ T_z^{\Kb} \partial \Omega, T_{z^-}^{\Kb} \partial \Omega\right\} \cap \left\{ T_{x^+_{\gamma}}^{\Kb} \partial \Omega, T_{x^-_{\gamma}}^{\Kb} \partial \Omega\right\} = \emptyset.
\end{align*}
Then by Lemma~\ref{lem:generic_endpoints} there exists bi-proximal elements $\gamma_k \in \Aut(\Omega)$ so that 
\begin{align*}
x^+_{\gamma_k} \rightarrow z \text{ and } x^-_{\gamma_k} \rightarrow x^+_{\gamma}.
\end{align*}
So by the previous lemma there exists a bi-proximal element $\varphi \in \Aut(\Omega)$ so that $x^+_{\varphi} = z$. But then by Proposition~\ref{prop:bi_prox_str}
\begin{align*}
T_z^{\Kb} \partial \Omega \cap \partial \Omega = \{z\}.
\end{align*}
\end{proof}

\begin{lemma}\label{lem:bi_prox_between}
Suppose $\Kb$ is either $\Cb$ or $\Hb$, $\Omega \subset\Pb(\Kb^{d+1})$ is a proper domain with $C^1$ boundary, and there exists $x,y \in \Lc(\Omega)$ such that $T_x^{\Kb} \partial \Omega \neq T_y^{\Kb} \partial \Omega$. Then for all $x^+, y^+ \in \Lc(\Omega)$ distinct there exists a bi-proximal element $\varphi \in \Aut(\Omega)$ so that $x^+=x^+_{\varphi}$ and $y^+=x^-_{\varphi}$. 
\end{lemma}

\begin{proof}
There exists $\phi_m, \varphi_n \in \Aut(\Omega)$ and $p \in \Omega$ so that $\phi_m p \rightarrow x^+$ and $\varphi_n p \rightarrow y^+$. By passing to a subsequence we may suppose that $\phi_m^{-1} p \rightarrow x^-$ and $\varphi_n^{-1} p \rightarrow y^-$ for some $x^-, y^- \in \partial \Omega$. 

Notice that if 
\begin{align*}
\left\{ T_{x^+}^{\Kb} \partial \Omega, T_{x^-}^{\Kb} \partial \Omega\right\} \cap \left\{ T_{y^+}^{\Kb} \partial \Omega, T_{y^-}^{\Kb} \partial \Omega\right\} = \emptyset
\end{align*} 
then the lemma follows from Lemma~\ref{lem:generic_endpoints} and Lemma~\ref{lem:att_fix_closed}. Motivated by this observation, the proof reduces to modifying $\phi_m$ and $\varphi_n$ so that $T_{x^-}^{\Kb} \partial \Omega \neq T_{y^-}^{\Kb} \partial \Omega$. 

Now by Theorem~\ref{thm:bi_prox_exist} and Proposition~\ref{prop:bi_prox_str} there exists a bi-proximal element $\gamma \in \Aut(\Omega)$ so that $\{ x^+, x^-, y^+, y^-\} \cap \{ x_{\gamma}^+, x_{\gamma}^-\} = \emptyset$. Then since $T_{x^\pm_{\gamma}}^{\Cb} \partial \Omega \cap \partial \Omega= \{ x^{\pm}_{\gamma} \}$ this implies that 
\begin{align*}
\left\{ T_x^{\Kb} \partial \Omega, T_{x^-}^{\Kb} \partial \Omega\right\} \cap \left\{ T_{x^+_{\gamma}}^{\Kb} \partial \Omega, T_{x^-_{\gamma}}^{\Kb} \partial \Omega\right\} = \emptyset
\end{align*}
So by Lemma~\ref{lem:generic_endpoints} and Lemma~\ref{lem:att_fix_closed} there exists a bi-proximal element $\phi$ so that $x^+ = x^+_{\phi}$ and $x_{\gamma}^+ = x^-_{\phi}$. Using the same argument we can find a bi-proximal element $\varphi$ so that $y^+ = x^+_{\varphi}$ and 
\begin{align*}
\{ x^+, x_{\gamma}^+\} \cap \{ y^+, x^+_{\varphi} \} = \emptyset. 
\end{align*}
Now 
\begin{align*}
\phi^n p \rightarrow x^+ , \phi^{-n} p \rightarrow x_{\phi}^-, \varphi^m p \rightarrow y^+, \text{ and } \varphi^{-m} p \rightarrow x_{\varphi}^-.
\end{align*}
Then by part (2) and (3) of Proposition~\ref{prop:bi_prox_str}
\begin{align*}
\left\{ T_{x^+}^{\Kb} \partial \Omega, T_{x_{\varphi}^-}^{\Kb} \partial \Omega\right\} \cap \left\{ T_{y^+}^{\Kb} \partial \Omega, T_{x_{\phi}^-}^{\Kb} \partial \Omega\right\} = \emptyset
\end{align*}
and so the lemma follows from Lemma~\ref{lem:generic_endpoints} and Lemma~\ref{lem:att_fix_closed}.
\end{proof}

\begin{proof}[Proof of Proposition~\ref{prop:limit_set}]
We first observe that $\Lc(\Omega)$ is closed. Suppose $x_n \in \Lc(\Omega)$ and $x_n \rightarrow x$. Then there exists $\varphi_{n,m} \in \Aut(\Omega)$ and $p_n \in \Omega$ so that 
\begin{align*}
\lim_{m \rightarrow \infty} \varphi_{n,m} p_n = x_n.
\end{align*}
Now fix $p \in \Omega$, then by Proposition~\ref{prop:limits} 
\begin{align*}
\lim_{m \rightarrow \infty} \varphi_{n,m} p = x_n.
\end{align*}
Then there exists $m_n \rightarrow \infty$ so that 
\begin{align*}
\lim_{n \rightarrow \infty} \varphi_{n,m_n} p = x.
\end{align*}
So $\Lc(\Omega)$ is closed. 

Now if $x, y \in \Lc(\Omega)$ then there exists a bi-proximal element $\varphi \in \Aut(\Omega)$ so that $x = x^+_{\varphi}$ and $y=x^-_{\varphi}$. Then by Corollary~\ref{cor:bi_3}, $y \in \Aut_0(\Omega) \cdot x$. Since $x,y \in \Lc(\Omega)$ were arbitrary we see that $\Aut_0(\Omega)$ acts transitively on $\Lc(\Omega)$. 

Now $\Aut_0(\Omega) \leq \PGL_{d+1}(\Cb)$ being a closed subgroup (see Proposition~\ref{prop:closed}) is a Lie subgroup and acts smoothly on $\Pb(\Kb^{d+1})$. Since $\Lc(\Omega) = \Aut_0(\Omega)\cdot x$ for any $x \in \Lc(\Omega)$ we see from Lemma~\ref{lem:smth_orbits} that $\Lc(\Omega)$ is a $C^\infty$ submanifold of $\Pb(\Kb^{d+1})$. 
\end{proof}

\section{Proof of Theorem~\ref{thm:C1}}

For this section suppose that $\Omega \subset \Pb(\Kb^{d+1})$ is a proper domain, $\partial \Omega$ is a $C^1$ hypersurface, and the limit set spans $\Kb^{d+1}$. Since the limit set spans there exists $x, y \in \Lc(\Omega)$ so that $x \notin T_y^{\Kb} \partial \Omega$. In particular, $T_x^{\Kb} \partial \Omega \neq T_y^{\Kb} \partial \Omega$. So $\Aut(\Omega)$ contains a bi-proximal element by Theorem~\ref{thm:bi_prox_exist}.

Now fix a bi-proximal element $\varphi \in \Aut_0(\Omega)$ and let $H^\pm =T_{x^\pm_{\varphi}}^{\Kb} \partial \Omega$. Pick coordinates so that 
\begin{enumerate}
\item $x^+_\varphi = [1:0:\dots:0]$,
\item $x^-_\varphi = [0:1:0: \dots :0]$,
\item $H^+ \cap H^- = \{ [0:0:z_2:\dots:z_d] : z_2, \dots, z_d \in \Kb\}$.
\end{enumerate}
For the rest of the proof identify $\Kb^d$ with the affine chart
\begin{align*}
\{ [1:z_1:z_2:\dots :z_d ] : z_1, \dots, z_d \in \Kb \}.
\end{align*}
Then by Theorem~\ref{thm:blow_up} there exists a $C^1$ function $F: \Kb^{d-1} \rightarrow \Omega$ so that 
\begin{align*}
\Omega = \{ (z_1,z_2, \dots, z_d) : \Real(z_1) > F(z_2, \dots, z_d) \}.
\end{align*}
and by Corollary~\ref{cor:special_bi} 
\begin{align*}
\psi_t : = \begin{pmatrix} e^t & & \\ & e^{-t} & \\ & & \Id_{d-1} \end{pmatrix}  \in  \Aut_0(\Omega)
\end{align*}
for $t \in \Rb$. 

Now there exists $x_1, \dots, x_{d-1} \in \Lc(\Omega)$ so that (as $\Kb$-lines)
\begin{align*}
x^+_\varphi  + x^-_\varphi  + x_1 + \dots + x_{d-1} = \Kb^{d+1}.
\end{align*}
By Proposition~\ref{prop:bi_prox_str}, $H^- \cap \partial \Omega = \{ x_{\varphi}^-\}$ and so $x_1, \dots, x_{d-1}$ are contained in our fixed affine chart.

Now we claim that 
\begin{align*}
T_0 \Lc(\Omega)  = \Kb_P \times \Kb^{d-1} = T_0 \partial \Omega.
\end{align*}
Notice that the second equality is by definition. 

For $1 \leq i \leq d-1$ let $L_i$ be the $\Kb$-line in $\Kb^d$ which contains $0$ and $x_i$.

Now fix some $i$. By Lemma~\ref{lem:bi_prox_between} and Theorem~\ref{thm:blow_up},  $L_i \cap \partial \Omega$ is projectively equivalent to a half space and thus in the affine chart $\Kb^d$ is either a half space or a open ball in $L_i$ (see Observation~\ref{obs:sphere_halfplane}). Since $F(z_2, \dots, z_d) > 0$ for all non-zero $(z_2, \dots, z_d)$ we see that $L_i \cap \partial \Omega$ must be an open ball in the affine chart. Moreover by Theorem~\ref{thm:blow_up}, $L_i \cap \partial \Omega \subset  \Lc(\Omega)$. Now since $L_i \cap \partial \Omega$ is a sphere we can pick $a_1,\dots, a_r \in L_i \cap \partial \Omega$ so that $r = \dimension_{\Rb} \Kb$ and
\begin{align*}
\Spanset_{\Rb} \{ a_1,\dots, a_r \} = L_i
\end{align*}
as elements of this affine chart. 

Let $P: \Kb^{d} \rightarrow \Kb^d$ be the projection 
\begin{align*}
P(z_1, \dots, z_d) = (0,z_2, \dots, z_d).
\end{align*}
Now if $z \in \Lc(\Omega)$ then
\begin{align*}
\psi_t(z) = (e^{-2t}z_1, e^{-t} z_2, \dots, e^{-t} z_d) \in \Lc(\Omega).
\end{align*}
and
\begin{align*}
\lim_{t \rightarrow \infty} \frac{1}{e^{-t}} \psi_t(z)  = (0, z_2, \dots, z_d) = P(z).
\end{align*}
Since $\Lc(\Omega)$ is a submanifold this implies that $P(z) \in T_0 \Lc(\Omega)$. Thus we see that $P(a_1), \dots, P(a_r) \subset T_0 \Lc(\Omega)$. Thus  $P(L_i) \subset T_0 \Lc(\Omega)$. 

Since $i$ was arbitrary we then see that 
\begin{align*}
P(L_1 + \dots + L_{d-1}) \subset T_0 \Lc(\Omega).
\end{align*}
But since 
\begin{align*}
x^+_\varphi  + x^-_\varphi  + x_1 + \dots + x_{d-1} = \Kb^{d+1}.
\end{align*}
as $\Kb$-lines we have
\begin{align*}
\{0\} \times \Kb^{d-1} = P(L_1 + \dots + L_{d-1}) \subset T_0 \Lc(\Omega).
\end{align*}
Using Theorem~\ref{thm:blow_up} we see that $\Kb_P \times \{ 0 \} \subset \Lc(\Omega)$ and so
\begin{align*}
T_0 \Lc(\Omega) =  \Kb_P \times \Kb^{d-1} = T_0 \partial \Omega.
\end{align*}
Thus $\Lc(\Omega) \subset \partial \Omega$ is an open and closed submanifold of $\partial \Omega$. Since $\partial \Omega$ is connected, this implies that $\Lc(\Omega) = \partial \Omega$. 

Then since $\Lc(\Omega)$ is a $C^\infty$ submanifold of $\Pb(\Kb^{d+1})$ we see that $\partial \Omega$ is $C^\infty$ and so the Theorem follows from Theorem~\ref{thm:C2}. 

\section{An example}\label{sec:example}

In this section we construct a non-symmetric proper domain $\Omega \subset \Pb(\Kb^{d+1})$ with $C^{1,1}$ boundary so that there exists $x,y \in \Lc(\Omega)$ with $T_x^{\Kb} \partial \Omega \neq T_y^{\Kb} \partial \Omega$.  

Let $\wh{F}: \Pb(\Kb^{d-1}) \rightarrow \Rb_{>0}$ be a $C^{1,1}$ function (which is not $C^2$) and define the function 
\begin{align*}
F(z_2, \dots, z_d) = \left\{ \begin{array}{ll} 
\wh{F}([z_2 : \dots  : z_d]) & \text{ if } (z_1, \dots, z_d) \neq 0 \\
0 & \text{ otherwise. }
\end{array} \right.
\end{align*}
Then consider the domain 
\begin{align*}
\Omega = \left\{ [1:z_1: \dots : z_d] \in \Pb(\Kb^{d+1}) : \Imaginary(z_1) > \left(\abs{z_2}^2 + \dots + \abs{z_{d}}^2\right) F(z_2,\dots ,z_d) \right\}.
\end{align*}
Clearly $\partial \Omega$ is $C^{1,1}$ away from $[1:0:\dots:0]$ and $[0:1:0:\dots:0]$. Since $F$ is bounded, $\partial \Omega$ is $C^{1,1}$ at $[1:0:\dots:0]$. Moreover if we consider the projective map 
\begin{align*}
T([z_0 : z_1 : z_2 : \dots : z_d])=  [ z_ 1 : -z_ 0 : z_2 : \dots : z_d]
\end{align*}
then $T(\Omega) = \Omega$. Thus $\Omega$ is also $C^{1,1}$ at $[0:1:0:\dots:0]$.

Finally notice that $\Aut(\Omega)$ contains the transformation 
\begin{align*}
[z_0: z_1:z_2\dots : z_d] \rightarrow [e^{t} z_0 : e^{-t} z_1 : z_2 : \dots : z_d]
\end{align*}
for any $t \in \Rb$. Thus $[1:0:\dots:0], [0:1:0:\dots:0] \in \Lc(\Omega)$.

\appendix 

\section{The Quaternions}\label{sec:quaternions}

In this expository section we will review the basic properties of the quaternions. The \emph{quaternions} $\Hb = \left\{ a+bi+cj+dk : a,b,c,d \in \Rb\right\}$ form a complex two dimensional vector space with multiplication rules:
\begin{align*}
i^2=j^2=k^2=ijk=-1.
\end{align*}
The quaternions have a natural conjugation:
\begin{align*}
\overline{a+bi+cj+dk} = a-bi-ci-dk
\end{align*}
and a corresponding absolute value:
\begin{align*}
\abs{a+bi+cj+dk}^2=(a+bi+cj+dk)\overline{(a+bi+cj+dk)} = a^2+b^2+c^2+d^2.
\end{align*}
One can also speak of the real part $\Real(z) = \frac{1}{2}(z+\overline{z})$ and the imaginary part $\Imag(z) = \frac{1}{2}(z-\overline{z})$ of a quaternion. 

In this paper we identify $\Hb^d$ with $d$-by-$1$ matrices with entries in $\Hb$ and let $\Hb$ act on $\Hb^d$ as follows:
\begin{align*}
\alpha \cdot (z_1, \dots, z_d)^t = (z_1 \alpha, \dots, z_d \alpha)^t.
\end{align*}
We then define $\GL_d(\Hb)$ to be the invertible $\Rb$-linear transformations of $\Hb^d$ which commute with the above action of $\Hb$.  If $M_d(\Hb)$ is the space of $d$-by-$d$ matrices with entries in $\Hb$, then we can identify
\begin{align*}
\GL_d(\Hb) = \GL_{2d}(\Cb) \cap M_d(\Hb).
\end{align*}
Since the quaternions are non-commutative this identification requires that the scalar multiplication acts on the right while $M_d(\Hb)$ acts on the left. 

Given $\varphi \in \GL(\Hb^d)$  we can define a determinant by viewing $\varphi$ as an element of $\GL_{2d}(\Cb)$:
\begin{align*}
D(\varphi) : = \abs{\det(\varphi: \Cb^{2d} \rightarrow \Cb^{2d})}.
\end{align*}
There are more sophisticated ways to define determinants for matrices with quaternionic entries, but the simple definition given above is good enough for our purposes. Finally, define the \emph{special linear group}
\begin{align*}
\SL_{d+1}(\Hb) = \{ \varphi \in \GL_{d+1}(\Hb) : D(\varphi)=1 \}.
\end{align*}

Now we can define the quaternionic projective space $\Pb(\Hb^{d+1})$ to be 
\begin{align*}
\Pb(\Hb^{d+1}) = \{ \vec{z} \in \Hb^{d+1} \} / \{ \vec{z} \sim \alpha \cdot \vec{z}\}.
\end{align*}
Then $\GL_{d+1}(\Hb)$ acts on $\Pb(\Hb^{d+1})$ and an element $\varphi \in \GL_{d+1}(\Hb)$ acts trivially if and only if 
\begin{align*}
\varphi = \begin{pmatrix} \alpha & & \\ & \ddots & \\ & & \alpha \end{pmatrix}
\end{align*}
for some $\alpha \in \Rb^*$. So the group 
\begin{align*}
\PGL_{d+1}(\Hb^{d+1}) = \GL_{d+1}(\Hb) / \{ \Rb^* Id \}
\end{align*}
acts faithfully on $\Pb(\Hb^{d+1})$. Moreover every $\varphi \in \PGL_{d+1}(\Hb)$ has a representative in $\SL_{d+1}(\Hb)$. 

\section{M{\"o}bius transformations}

In this section we review the basic properties of M{\"o}bius transformations when $\Kb$ is either $\Cb$ or $\Hb$. All these facts are well known when $\Kb=\Cb$.

We can identify $\Pb(\Kb^2)$ with $\overline{\Kb} = \Kb \cup \{\infty\}$ via the map
\begin{align*}
[z_1 : z_2 ] \rightarrow \left\{\begin{array}{ll} 
z_1(z_2)^{-1} & \text{ if } z_2 \neq 0\\
\infty & \text{ otherwise.}
\end{array} \right.
\end{align*}

With this identification $\PGL_2(\Kb)$ acts on $\overline{\Kb}$ by 
\begin{align*}
\begin{pmatrix} a & b \\ c & d \end{pmatrix} \cdot z = (az+b)(cz+d)^{-1}.
\end{align*}

As in the complex case,  M{\"o}bius transformations map spheres and hyperplanes to spheres and hyperplanes. 

\begin{observation} 
\label{obs:sphere_halfplane}
$\PGL_2(\Kb)$ maps spheres and hyperplanes to spheres and hyperplanes. 
\end{observation}

\begin{proof}
Every sphere and half plane can be described as a set of the form
\begin{align*}
\{ z \in \Kb : \abs{z-a} = R \abs{z-b}\}
\end{align*}
for some $a, b \in \Kb$ and $R >0$. Moreover every set of this form is a sphere or half plane. A calculation shows that M{\"o}bius transformations map a set of this form to a set of this form.
\end{proof}

Let 
\begin{align*}
\Hc_+ = \{ z \in \Kb : \Real(z) > 0\}.
\end{align*}
Now $\Hc_+$ is projectively equivalent to the unit ball by the M{\"o}bius transformation
\begin{align*}
z \rightarrow (z-1)(z+1)^{-1}.
\end{align*}
In particular, $\Aut(\Hc_+)$ is isomorphic with
\begin{align*}
\Aut( \{ \abs{z} < 1\} )  = \PU_{\Kb}(1,1) = \{ \varphi \in \PGL(\Kb^2) : Q \circ \varphi = Q \} 
\end{align*}
where $Q(z) = \abs{z_1}^2 - \abs{z_2}^2$. The next proposition follows from the basic geometry of rank one symmetric spaces of non-compact type, but we provide an elementary proof.

\begin{proposition} \label{prop:auto_halfspace} \
\begin{enumerate}
\item If $x \in \partial \Hc_+ \subset \overline{\Kb}$ then the group 
\begin{align*}
P_x = \{ \varphi \in \Aut_0(\Hc_+) : \varphi x = x\}
\end{align*}
 acts transitively on $\Hc_+$,
\item $\Aut_0(\Hc_+)$ acts transitively on $\partial \Hc_+$,
\item $\Aut_0(\Hc_+)$ is generated by the two subgroups
\begin{align*}
U = \left\{ \begin{pmatrix} 1 & w \\ 0 & 1 \end{pmatrix} : \Real(w)=0 \right\} \text{ and } V = \left\{ \begin{pmatrix} 1 & 0 \\ w & 1 \end{pmatrix} : \Real(w)=0 \right\}.
\end{align*}
\end{enumerate}
\end{proposition}

\begin{proof}
A direct calculation shows that 
\begin{align*}
P_\infty = \left\{ \begin{pmatrix} \lambda & w \\ 0 & \overline{\lambda}^{-1}  \end{pmatrix} : \lambda, w \in \Kb, \ \lambda \neq 0, \ \Real(w) =0\right\}.
\end{align*}
Then $P_{\infty}$ clearly acts on transitively on $\Hc_+$ and $\partial \Hc_+ \setminus \{\infty\}$. Since 
\begin{align*}
P_0 =   \begin{pmatrix} 0 & 1 \\ 1 & 0 \end{pmatrix}^{-1}  P_{\infty} \begin{pmatrix} 0 & 1 \\ 1 & 0 \end{pmatrix}
\end{align*}
we see that $P_0$ acts transitively on $\partial \Hc_+ \setminus \{0\}$. Since $\Aut_0(\Hc_+)$ contains $P_0$ and $P_{\infty}$, this implies part (2). Then since $\Aut_0(\Hc_+)$ acts transitively on the boundary, we see that every group $P_x$ is conjugate to $P_{\infty}$. Then since $P_{\infty}$ acts transitively on $\Hc_+$, we then have part (1). 

It remains to prove part (3). Let $G$ be the closed group generated by $U$ and $V$. If $\Real(u)=\Real(w)=0$ then
\begin{align*}
\left[ \begin{pmatrix} 0 & w \\ 0 & 0 \end{pmatrix}, \begin{pmatrix} 0 & 0 \\ u & 0 \end{pmatrix} \right] = \begin{pmatrix} wu & 0 \\ 0 & -uw \end{pmatrix} =  \begin{pmatrix} wu & 0 \\ 0 & \overline{wu} \end{pmatrix}.
\end{align*}
So the Lie algebra of $G$ contains 
\begin{align*}
\left\{ \begin{pmatrix} \lambda & w \\ u & -\overline{\lambda} \end{pmatrix} : \lambda, w, u \in \Kb, \ \Real(w)=\Real(u)=0 \right\}.
\end{align*}
In particular $G$ contains $P_\infty$ and $P_0$. This implies that $G$ acts transitively on the boundary. Now suppose $\varphi \in \Aut_0(\Hc_+)$. Since $G$ acts transitively on $\partial \Hc_+$ there exists $\gamma \in G$ such that $(\gamma \varphi)(0) = 0$. Then $\gamma \varphi \in P_0 \subset G$ which implies that $\varphi \in G$. 
\end{proof}

 \bibliographystyle{alpha}
\bibliography{hilbert}

\end{document}